\documentclass[%
aip,
jmp,
 amsmath,amssymb,
reprint, onecolumn 
]{revtex4-1}

\usepackage{graphicx}
\usepackage{dcolumn}
\usepackage{bm}

\usepackage[utf8]{inputenc}
\usepackage[T1]{fontenc}
\usepackage{mathptmx}
\usepackage{etoolbox}

\usepackage{eurosym}
\usepackage{amssymb,amsfonts,amsmath,times}
\usepackage{color}
\usepackage{xcolor}

\numberwithin{equation}{section}
\newtheorem{theorem}{Theorem}[section]

\newtheorem{assumption}[theorem]{Assumption}

\newtheorem{corollary}[theorem]{Corollary}

\newtheorem{definition}[theorem]{Definition}
\newtheorem{example}[theorem]{Example}

\newtheorem{lemma}[theorem]{Lemma}

\newtheorem{remark}[theorem]{Remark}

\newenvironment{proof}[1][Proof]{\noindent\textbf{#1.} }{\ \rule{0.5em}{0.5em}}

\newcommand{\reals}{\mathbb{R}}

\newcommand{\timeint}{\mathcal{T}}

\newcommand{\coloneqq}{:=}
\renewcommand{\cite}[2][]{[\onlinecite[#1]{#2}]}

\renewcommand\theequation{\arabic{section}.\arabic{equation}}

\newcommand{\GeorgyMarkup}[1]{\begin{color}{black}#1\end{color}}
\newcommand{\GeorgyMarkupy}[1]{\begin{color}{black}#1\end{color}}
\newcommand{\GeorgyMarkupx}[1]{\begin{color}{black}#1\end{color}}

\newcommand{\AlexMarkup}[1]{\begin{color}{black}#1\end{color}}

\makeatletter
\def\@email#1#2{%
 \endgroup
 \patchcmd{\titleblock@produce}
  {\frontmatter@RRAPformat}
  {\frontmatter@RRAPformat{\produce@RRAP{*#1\href{mailto:#2}{#2}}}\frontmatter@RRAPformat}
  {}{}
}%
\makeatother
\begin{document}

\preprint{AIP/123-QED}

\title[Stochastic dynamics of particle systems on unbounded degree graphs]{Stochastic dynamics of particle systems on unbounded degree graphs}
\author{Georgy Chargaziya}
 \altaffiliation{Contact email address - georgy.chargaziya@swansea.ac.uk}
 \affiliation{Mathematics Department, Swansea University, Swansea, West Glamorgan, SA1 8EN, United Kingdom.}
\author{Alexei Daletskii}%
 \altaffiliation{Contact email address - alex.daletskii@york.ac.uk}
 \affiliation{Mathematics Department, University of York, York, North Yorkshire, YO10 5DD, United Kingdom.}

\date{\today}

\begin{abstract}
We consider an infinite system of coupled stochastic differential equations
(SDE)\ describing dynamics of the following infinite particle system. Each
particle is characterised by its position $x\in \mathbb{R}^{d}$ and
internal parameter (spin) $\sigma _{x}\in \mathbb{R}$. While the positions
of particles form a fixed ("quenched") locally-finite set (configuration) $%
\gamma \subset $ $\mathbb{R}^{d}$, the spins $\sigma _{x}$ and $\sigma _{y}$
interact via a pair potential whenever $\left\vert x-y\right\vert <\rho $,
where $\rho >0$ is a fixed interaction radius. The number $n_{x}$ of
particles interacting with a particle in position $x$ is finite but
unbounded in $x$. The growth of $n_{x}$ \GeorgyMarkup{as $|x|\rightarrow \infty $} creates a
major technical problem for solving our SDE system. To overcome this
problem, we use a finite volume approximation combined with a version of the
Ovsjannikov method, and prove the existence and uniqueness of the solution
in a scale of Banach spaces of weighted sequences. As an application
example, we construct stochastic dynamics associated with Gibbs states of
our particle system.\\
\textbf{Keywords:} interacting particle systems, infinite systems of
stochastic equations, scale of Banach spaces, Ovsjannikov's method,
dissipativity.\\
\textbf{2010 Mathematics Subject Classification:} 82C20; 82C31; 60H10; 46E99
\end{abstract}
\maketitle
\section{Introduction}

In recent decades, there has been an increasing interest in studying
countable systems of particles randomly distributed in the Euclidean space $%
\mathbb{R}^{d}$. In such systems, each particle is characterized by its
position $x\in X:=\mathbb{R}^{d}$ and an internal parameter (spin) $\sigma
_{x}\in S:=\mathbb{R}^{n}$, see for example \cite{Romano}, \cite[Section 11]%
{OHandley}, \cite{Bov} and \cite{DKKP,DKKP1}, pertaining to modelling of
non-crystalline (amorphous) substances, e.g. ferrofluids and amorphous
magnets. Throughout the paper we suppose, mostly for simplicity, that $n=1$.

Let us denote by $\Gamma (X)$ the space of all locally finite subsets
(configurations) of $X$ and consider a particle system with positions
forming a given fixed (\textquotedblleft quenched\textquotedblright )
configuration $\gamma \in \Gamma (X)$. Two spins $\sigma _{x}$ and $\sigma
_{y} $, $x,y\in\gamma$, are allowed to interact via a pair potential if the
distance between $x$ and $y$ is no more than a fixed interaction radius $%
\rho >0$, that is, they are neighbours in the geometric graph defined by $%
\gamma $ and $\rho $. The evolution of spins is described then by a system
of coupled stochastic differential equations.

Namely, we consider, for a fixed $\gamma \in \Gamma (X)$, a system of
stochastic differential equations in \GeorgyMarkup{$S=\mathbb{R}$} of the following
form: 

\begin{equation}
d\xi _{x,t}=\Phi _{x}(\Xi _{t})dt+\Psi _{x}(\Xi _{t})dW_{x,t},\ \quad x\in
\gamma ,\ t\geq 0,  \label{system}
\end{equation}%
where $\Xi _{t}=\left( \xi _{x,t}\right) _{x\in \gamma }$ and $\left(
W_{x,t}\right) _{x\in \gamma }$ are, respectively, families of real-valued
stochastic processes and independent Wiener processes on a suitable
probability space. Here the drift and diffusion coefficients $\Phi _{x}$ and 
$\Psi _{x}$ are real-valued functions, defined on the Cartesian power $%
S^{\gamma }:=\left\{ \bar{\sigma}=(\sigma _{x})_{x\in \gamma }\left\vert
\sigma _{x}\in S,x\in \gamma \right. \right\} $. Both $\Phi _{x}$ and $\Psi
_{x}$ are constructed using pair interaction between the particles and their
self-interaction potentials, see Section \ref{sec-notanddefs}, and are
independent of $\sigma _{y}$ if $\left\vert y-x\right\vert >\rho $.

The aim of including the diffusion term in (\ref{system}) is two-fold. On
the one hand, it allows to consider the influence of random forces on our
particle system and, on the other hand, to construct and study stochastic
dynamics associated with the equilibrium (Gibbs) states of the system. The
Gibbs states of spin systems on unbounded degree graphs have been studied in 
\cite{KKP,DKKP,DKKP1}, see also references given there.

The case where vertex degrees of the graph are globally bounded (in
particular, if $\gamma $ has a regular structure, e.g. $\gamma =\mathbb{Z}%
^{d}$) has been well-studied (in both deterministic and stochastic cases),
see e.g. \cite{DoRo,Fri,HoSt,LeRit,Roy,Yosh,LLL,DZ1,AKT,AKRT,ADK,ADK1,INZ}, and references therein.
However, the aforementioned applications to non-crystalline substances require
dealing with unbounded vertex degree graphs. An important example of such
graphs is served by configurations $\gamma $ distributed according to a
Poisson or, more generally, Gibbs measure on $\Gamma(X)$ with a superstable
low regular interaction energy, in which case the typical number of
\textquotedblleft neighbours\textquotedblright\ of a particle located at $%
x\in X$ is proportional to $\sqrt{1+\log |x|}$, see e.g. \cite{R70} and \cite%
[p.1047]{K93}.

There are two main technical difficulties in the study of system (\ref%
{system}). The first one is related to the fact that the number of particles
interacting with a tagged particle $x$ is finite but unbounded in $x\in
\gamma $. Consequently, the system cannot be considered an equation
in a fixed Banach space and studied by standard methods of e.g. \cite{DaFo}, 
\cite{DZ}.

\begin{color}{black}%
The way around it has been proposed in \GeorgyMarkup{\cite{DaF}}, where a deterministic
version of system (\ref{system}) (with $\Psi \equiv 0$) was considered in an
expanding scale of embedded Banach spaces of weighted sequences and solved
using a version of the Ovsjannikov method.

Originally, the Ovsjannikov method was developed for a linear equation 
\begin{equation}
\dot{X}_{t}=AX_{t}  \label{lin1}
\end{equation}%
in a scale of densely embedded Banach spaces $B_{\alpha }$, $\alpha \in 
\mathcal{A}$, where $\mathcal{A}$ is a real interval, such that $B_{\alpha
}\subset B_{\beta }$ if $\alpha <\beta $, and $A$: $B_{\alpha }\rightarrow
B_{\beta }$ is bounded with norm satisfying the estimate%
\begin{equation}
\left\Vert A\right\Vert _{\alpha ,\beta }\leq L(\beta -\alpha )^{-q},\
\alpha <\beta ,\ q=1,  \label{Ovs-0}
\end{equation}%
for any $\alpha ,\beta \in \mathcal{A}$. Then, for $X_{0}\in B_{\alpha }$,
equation (\ref{lin1}) has a solution $X_{t}\in B_{\beta }$, $t<T$, for
finite $T$ depending on $\alpha $ and $\beta $.

It was noticed in \GeorgyMarkup{\cite{DaF}} that, under a stronger norm bound with $q<1$
in (\ref{Ovs-0}), the lifetime of the solution $X_{t}\in B_{\beta }$ is
infinite. That fact allows to find a global uniform bound for a sequence of
finite volume approximations of the system of differential equations in
question and prove its convergence, thus proving the existence and
uniqueness of the global solution of the deterministic version of (\ref{system}%
). 

The first advances in the study of stochastic equations in the scale $%
\left\{ B_{\alpha },\alpha \in \mathcal{A}\right\} $, were made in \cite{Dal}
and \cite{ADGC}, where, respectively, local and global strong solutions of a
general stochastic equation had been constructed. In those works, the
coefficients are assumed to be Lipschitz mappings $B_{\alpha }\rightarrow
B_{\beta }$ for any $\alpha <\beta $, with Lipschitz constants $L(\beta
-\alpha )^{-q}$, $q=\frac{1}{2}$ and $q<\frac{1}{2}$, respectively. Observe
that the threshold value of $q$ here is $\frac{1}{2}$ instead of $1$ as in (%
\ref{Ovs-0}) because of the presence of the Itô integral, which makes it
necessary to work in $L^{2}$ spaces instead of $L^{1}$.

The results of \cite{Dal} and \cite{ADGC} are applicable to system (\ref%
{system}) only in the case where the drift coefficients $\Phi _{x},\ x\in
\gamma $, are globally Lipschitz. However, to construct the
dynamics associated with Gibbs states of interacting particle systems, one
has to consider the drift coefficients that are only locally Lipschitz. The
existence of such dynamics, under certain dissipativity conditions on the
drift, is known in the situation of a regular lattice, see \cite{AKT,AKRT}
(observe that those works deal with the more complicated quantum systems but
are applicable to classical systems, too, albeit only for the additive
noise).

For deterministic systems on unbounded degree graphs, the dissipative case
was considered in the aforementioned paper \GeorgyMarkup{\cite{DaF}}. In the present
work, we revisit the volume approximation approach of that paper. However,
the presence of stochastic terms requires the application of very different
techniques. To prove the convergence of finite volume
approximations, we have developed a version of the Gronwall inequality
suitable for a scale of Banach spaces. In this way, we have been able to
prove the existence and uniqueness of global strong solutions of (\ref%
{system}) and their component-wise time continuity, in the case of
dissipative single-particle potentials. 

\end{color}%

The structure of the paper is as follows. In Section \ref{sec-notanddefs} we
introduce the framework and formulate our main results. Section \ref%
{proof-existence} is devoted to the proof of the existence and uniqueness
result for (\ref{system}). In a short Section \ref{Markov}, we discuss
Markov semigroup generated by the solution of (\ref{system}). In Section \ref%
{sdyn}, we study stochastic dynamics associated with Gibbs states of our
system.

Finally, the Appendix contains auxiliary results on linear operators in the
scales of Banach spaces, estimates of the solutions of system (\ref{system})
and, notably, a crucial for our techniques generalization of the classical
comparison theorem and a Gronwall-type inequality, suitable for our
framework.

\textbf{Acknowledgment. } We are very grateful to Zdzislaw Brzezniak,
Dmitry Finkelshtein, Yuri Kondratiev and Jiang-Lun Wu for their interest in
this work and stimulating discussions.

\section{The setup and main results \label{sec-notanddefs}\label%
{sec-stochsystem}}

\begin{color}{black}

Let us fix a configuration $\gamma\in\Gamma(X)$ and a family $\left(
W_{x,t}\right) _{x\in \gamma }$ of independent Wiener processes on a suitable filtered and complete probability space $\mathbf{P}%
\coloneqq%
(\Omega ,\mathcal{F},\mathbb{F},\mathbb{P})$.
Our aim is to find a strong solution of SDE system (\ref{system}), that is, a family
$\Xi _{t}=\left( \xi _{x,t}\right) _{x\in \gamma }$ of continuous adapted stochastic processes on $\mathbf{P}$ such that the equality 

\begin{align}
\xi _{x,t}& =\zeta _{x}+\int_{0}^{t}\Phi _{x}(\Xi _{s})ds+\int_{0}^{t}\Psi
_{x}(\Xi _{s})dW_{x,s},\ x\in \gamma , \,
 \zeta _{x}\in S,  \label{MainSystem}  
\end{align}
holds for all $ t\in \mathcal{T}:=[0,T]$, $T>0$, almost surely, that is, on a common for all $t$ set of probability $1$.
The coefficients $\Phi _{x}$
and $\Psi _{x}$ are defined explicitly in Assumption \ref{mainass} below, and $\int_{0}^{t}\Psi
_{x}(\Xi _{s})dW_{x,s}$ is the continuous version of the Ito integral, cf. Remark \ref{RII.4}.

\end{color}

First, we need to introduce some notations. We fix $\rho >0$ and denote by $%
n_{x}$, $x\in \gamma $, the number of elements in the set 
\begin{equation*}
\bar{\gamma}_{x}:=\left\{ y\in \gamma :\left\vert x-y\right\vert \leq \rho
\right\} .
\end{equation*}%
Observe that $n_{x}\geq 1$ for all $x\in \gamma $, because $x\in \bar{\gamma}%
_{x}$. We will also use the notation $\gamma _{x}:=\bar{\gamma}_{x}\setminus
\left\{ x\right\} \equiv \left\{ y\in \gamma :\left\vert x-y\right\vert \leq
\rho ,y\neq x\right\} $.

For a fixed $\gamma \in \Gamma (X)$, we will consider the Cartesian product $%
S^{\gamma }$ of identical copies $S_{x}$, $x\in \gamma $, of $S$, and denote
its elements by $\bar{z}:=\left( z_{x}\right) _{x\in \gamma }$, etc. When
dealing with multiple configurations $\eta \in \Gamma (X)$, we will
sometimes write $\bar{z}_{\eta }:=\left( z_{x}\right) _{x\in \eta }$, to
emphasize the dependence on $\eta $.

We will work under the following assumption.

\begin{assumption}
\label{mainass}
\end{assumption}

\begin{enumerate}
\item[(\textbf{A})] There exists a constant $C>0$ such that 
\begin{equation}
n_{x}\leq C(1+\log (1+|x|))\ \text{for all }x\in \gamma.  \label{logbound}
\end{equation}

\GeorgyMarkup{
\item[(\textbf{B})] The drift coefficients $\Phi _{x},\ x\in \gamma $, have
the form%
\begin{equation}
\Phi _{x}(\bar{z})%
\coloneqq%
\phi (z_{x})+\sum_{y\in \bar{\gamma} _{x}}\varphi_{x,y}(z_{x},z_{y}),\text{for all }x\in \gamma,
\label{Phi-form}
\end{equation}%
where $\phi :S\rightarrow S$ is a measurable function and $\varphi _{xy}:S^{2}\rightarrow S$ are also measurable functions
satisfying uniform Lipschitz condition 
\begin{align}
	\left\vert \varphi _{xy}(\sigma _{1},s_{1})-\varphi  _{xy}(\sigma
	_{2},s_{2})\right\vert & \leq \bar{a}\left( \left\vert \sigma _{1}-\sigma
	_{2}\right\vert +\left\vert s_{1}-s_{2}\right\vert \right), \nonumber \\
	\left\vert \varphi _{xy}(\sigma _{1},s_{1}))\right\vert &\leq \bar{a}\left(1 + \left\vert \sigma _{1}\right\vert +\left\vert s_{1}\right\vert \right), \nonumber
\end{align}
for some constant $\bar{a}>0$ and all $x,y\in \gamma ,\ \sigma _{1},\sigma
_{2},s_{1},s_{2}\in S.$
}
\item[(\textbf{C})] There exist constants $c>0$ and $R\geq2$ such that 
\begin{equation}
|\phi (\sigma )|\leq c(1+|\sigma |^{R}),\ \sigma \in S.  \label{bound1}
\end{equation}

\item[(\textbf{D})] There exists $b>0$ such that 
\begin{equation}
(\sigma _{1}-\sigma _{2})(\phi (\sigma _{1})-\phi (\sigma _{2}))\leq
b(\sigma _{1}-\sigma _{2})^{2},\ \sigma _{1},\sigma _{2}\in S.
\label{cond-diss1}
\end{equation}

\item[(\textbf{E})] The diffusion coefficients $\Psi _{x},\ x\in \gamma $,
have the form 
\begin{equation}
\Psi _{x}(\GeorgyMarkup{\bar{z}})%
\coloneqq%
\sum_{y\in \bar{\gamma}_{x}}\psi _{xy}(z_{x},z_{y})\ \text{for all }x\in
\gamma ,  \label{Psi-form}
\end{equation}%
where $\psi _{xy}:S^{2}\rightarrow S$ are measurable functions
satisfying uniform Lipschitz condition 
\GeorgyMarkup{
\begin{align}
\left\vert \psi _{xy}(\sigma _{1},s_{1})-\psi _{xy}(\sigma
_{2},s_{2})\right\vert &\leq M\left( \left\vert \sigma _{1}-\sigma
_{2}\right\vert +\left\vert s_{1}-s_{2}\right\vert \right) , \nonumber \\
\left\vert \psi _{xy}(\sigma _{1},s_{1}))\right\vert &\leq M\left(1 + \left\vert \sigma _{1}\right\vert +\left\vert s_{1}\right\vert \right) ,
\label{cond-lip1}
\end{align}
}
for some constant $M>0$ and all $x,y\in \gamma ,\ \sigma _{1},\sigma
_{2},s_{1},s_{2}\in \GeorgyMarkup{S}.$
\end{enumerate}

\bigskip

The specific form of the coefficients requires the development of a special
framework. Indeed, we will be looking for a solution of (\ref{MainSystem})
in a scale of expanding Banach spaces of weighted sequences, which we
introduce below.

\GeorgyMarkup{We start with a general definition and consider a family $\mathfrak{B}$ of
Banach spaces $B_{\alpha }$ indexed by $\alpha \in \bar{\mathcal{A}}:=[\alpha
_{\ast },\alpha ^{\ast }]$ with fixed $0\leq \alpha _{\ast },\alpha ^{\ast
}<\infty $, and denote by $\left\Vert \cdot \right\Vert _{B_{\alpha }}$ the
corresponding norms. When speaking of these spaces and related objects, we
will always assume that the range of indices is $[\alpha _{\ast },\alpha
^{\ast }]$, unless stated otherwise. The interval $\bar{\mathcal{A}}$ remains
fixed for the rest of this work. We will also use the corresponding semi-open
interval $\mathcal{A}:=[\alpha _{\ast },\alpha ^{\ast })$.}

\begin{definition}
\label{Defscale}The family $\mathfrak{B}$ is called a scale if 
\begin{equation*}
B_{\alpha }\subset B_{\beta }\ {\text{and }}\left\Vert u\right\Vert
_{B_{\beta }}\leq \left\Vert u\right\Vert _{B_{\alpha }}{\text{ for any }}%
\alpha <\beta ,\ u\in B_{\alpha },\ \alpha ,\beta \in \mathcal{\bar{A}},
\end{equation*}%
where the embedding means that $B_{\alpha }$ is a dense vector subspace of $%
B_{\beta }$.
\end{definition}

For any $\ \alpha ,\beta \in \mathcal{A}$, we will use the notation 
\begin{equation*}
B_{\alpha +}%
\coloneqq%
{\bigcap }_{\beta >\alpha }B_{\beta }.
\end{equation*}

The two main scales we will be working with are given by the spaces $%
l_{\alpha }^{p}$ of weighted sequences and $l_{\alpha }^{p}$-valued random
processes, respectively, defined as follows.

\begin{color}{black}%

\begin{itemize}
\item[(1)] For all $p\geq 1$ and $\alpha \in \mathcal{\bar{A}}$ let 
\begin{align}
& l_{\mathfrak{\alpha }}^{p}%
\coloneqq%
\left\{ \bar{z}\in S^{\gamma }\ \left\vert \ \Vert \bar{z}\Vert _{l_{\alpha
}^{p}}%
\coloneqq%
\left( \sum_{x\in \gamma }\GeorgyMarkup{w(x)}^{-1}|z_{x}|^{p}\right) ^{\frac{1}{p}%
}<\infty \right. \right\} ,  \label{l-scale} \\[0.01in]
& \GeorgyMarkup{w(x) =} e^{a|x|},  \notag
\end{align}%
and $\mathcal{L}^{p}%
\coloneqq%
\{l_{\alpha }^{p}\}_{\mathfrak{\alpha }\in \mathcal{A}}$ be, respectively, a
Banach space of weighted real sequences and the scale of such spaces.
\end{itemize}

\end{color}%

\begin{itemize}
\item[(2)] \GeorgyMarkupx{For all $p\geq 1$ and $\alpha \in \mathcal{\bar{A}}$ let $%
\mathcal{R}_{\alpha }^{p}$ denote the Banach space of $l_{\alpha }^{p}$
-valued random processes $\bar{\xi}_{t},\ t\in \mathcal{T}$, on
probability space $\mathbf{P}$, with progressively measurable components and finite norm%
\begin{equation*}
\Vert \bar{\xi}\Vert _{\mathcal{R}_{\alpha }^{p}}%
\coloneqq%
\left( \sup \left\{ \mathbb{E}\Vert \GeorgyMarkup{\bar{\xi} _{t}}\Vert _{l_{\alpha }^{p}}^{p}\
\left\vert \ t\in \mathcal{T}\right. \right\} \right) ^{\frac{1}{p}}<\infty ,
\end{equation*}%
and let $\mathcal{R}^{p}%
\coloneqq%
\{\mathcal{R}_{\alpha }^{p}\}_{\alpha \in \mathcal{A}}$ be the scale of such
spaces.}
\end{itemize}

\begin{remark}
\begin{color}{black}%
The choice of exponential weights in the definition of space $l_{\alpha }^{p}
$ is dictated by the logarithmic growth condition on numbers $n_{x}$, cf. %
\ref{logbound}, which in turn is motivated by the fact that it holds for a
typical configuration $\gamma $ distributed according to a Poisson or, more
generally, Gibbs measure on $\Gamma (X)$ with a superstable low regular
interaction energy, in which case $n_{x}$ is proportional to $\sqrt{1+\log
|x|}$, see e.g. \cite{R70} and \cite[p.1047]{K93}. In general, an informal
balance condition between $n_{x}$ and $w(|x|)$ is given by $w(|x|)\approx
exp(exp(n_{x}))$, see Sect. 2.2 of \cite{DaF} for details.%
\end{color}%
\end{remark}
\begin{remark}\label{RII.4}
\GeorgyMarkupx{Note that for $p\geq 2$, the definition of norms in $\mathcal{R}_{\alpha}^{p}$ and $l^{p}_{\alpha}$ implies that for any $\bar{\xi}\in\mathcal{R}_{\alpha}^{p}$ and any $x\in\gamma$ we have
$\mathbb{E}[\int_{0}^{T}\xi_{x,t}^{2}dt] < \infty$.
Moreover, since each component of $\bar{\xi}$ is progressively measurable, from the classical theory of integration with respect to the standard Wiener process we see that for all $x\in \gamma$ the integral
$\int_{0}^{t}\xi_{x,s}dW_{s}$ is well defined and so is the integral
$\int_{0}^{t}\Psi_{x}(\bar{\xi}_{s})dW_{x,s}$,
because $\Psi_{x}$ is a finite sum of measurable uniformly Lipschitz functions.}
\begin{color}{black}Moreover, the process $\int_{0}^{t}\Psi_{x}(\bar{\xi}_{s})dW_{x,s}$, $ t\in \mathcal{T}$, has a (unique) continuous version. \end{color}
\end{remark}

\GeorgyMarkup{For all $p\geq 1$ and $\alpha \in \mathcal{\bar{A}}$ we let 
\begin{align}
	L^{p}_{\alpha} \equiv L^{p}(\Omega,l_{\alpha }^{p}) \coloneqq \bigg{\{}X:\Omega\to l_{\alpha }^{p}\ | \  \bigg{(}\mathbb{E} \bigg{[}\|X\|_{l_{\alpha }^{p}}^{p} \bigg{]} \bigg{)}^{\frac{1}{p}} < \infty  \bigg{\}} \nonumber
\end{align}	}
be the space of $l_{\alpha }^{p}$-valued $p$-integrable random variables.

\bigskip
Our main result is the following theorem.

\begin{theorem}
\label{Existence} Suppose that Assumption \ref{mainass} holds. Then, for all 
$p\geq R$ \GeorgyMarkup{and any $\mathcal{F}_{0}$-measurable $\bar{\zeta}
:=(\zeta _{x})_{x\in \gamma }\in L_{\alpha }^{p}$}, $\alpha\in\mathcal{A}$, stochastic system (\ref%
{MainSystem}) admits a unique \begin{color}{black}(up to indistinguishability) 
strong  solution $\Xi \in \mathcal{R}_{\alpha +}^{p}$.
Moreover, the map%
\begin{equation*}
L^{p}_{\alpha}\ni \bar{\zeta}\mapsto \Xi \in \mathcal{R}_{\beta }^{p}
\end{equation*}%
is continuous for any $\beta >\alpha$.  
\end{color}
\end{theorem}

\begin{remark}
 Assumption $p\geq R$ ensures that given $\bar{\xi}\in \mathcal{R}_{\beta
}^{p}$ the random variable $\phi (\bar{\xi}_{t})$ is integrable for any $t\geq 0$.
\end{remark}

The proof of Theorem \ref{Existence} will be given in Section \ref%
{proof-existence}.

\bigskip

Our second main result is about the construction of non-equilibrium
stochastic dynamics associated with Gibbs states of our system. We consider
a Gibbs measure $\nu $ on $S^{\gamma }$ defined by the pair interaction $%
W_{x,y}(\sigma _{x},\sigma _{y})=a(x-y)\sigma _{x}\sigma _{y}$, $\sigma
_{x},\sigma _{y}\in S$, $x,y\in \gamma $, where \GeorgyMarkup{$a:X\to \reals$ is a measurable function with compact support} and a single particle potential $V:\mathbb{R}\rightarrow \mathbb{R}$
satisfying the lower bound%
\begin{equation*}
V(\sigma )\geq a_{V}\left\vert \sigma \right\vert ^{R+\varepsilon }-b_{V},\
\sigma \in S,\ \text{for some }a_{V},b_{V}>0\text{ and }\varepsilon >0,
\end{equation*}%
which is supported on $l_{\alpha }^{p}$ for some $\alpha \in \mathcal{A}$
and $p\in \left[R,R+\varepsilon \right] $, see
Section \ref{gibbs0} for details. Suppose now that $\phi $ in (\ref{Phi-form}%
) has a gradient form, that is, $\phi =-\nabla V$, and $\left\{ 
\begin{array}{c}
\psi _{xy}=0,\ x\neq y \\ 
\psi _{xx}=1%
\end{array}%
\right. $ for all $x,y\in \gamma $, so that our noise is additive, cf. (\ref%
{Psi-form}). Let ${\rm T}_{t}$ be the Markov semigroup defined by the process $\Xi
_{t}$ in a standard way. This semigroup acts in the space $C_{b}(l_{\alpha
+}^{p})$ of bounded continuous functions on space $l_{\alpha +}^{p}=\cap
_{\beta >\alpha }l_{\beta }^{p}$ equipped with the projective limit
topology, see Section \ref{Markov} below for details.

\begin{theorem}
\begin{color}{black}%
Gibbs measure $\nu $ is a symmetrizing (reversible) distribution for the
solution of (\ref{MainSystem}), that is,%
\begin{equation*}
\int {\rm T}_{t}f(\bar{\zeta})g(\bar{\zeta})\nu (d\bar{\zeta})=\int f(\bar{\zeta}%
){\rm T}_{t}g(\bar{\zeta})\nu (d\bar{\zeta}),\, t\ge 0,
\end{equation*}%
for any $\alpha \in \mathcal{A}$ and $f,g\in C_{b}(l_{\alpha +}^{p})$.%
\end{color}%
\end{theorem}

The proof of this result will be given in Section \ref{sec-stoch-dyn}.

\bigskip

From now on, the constant $p\geq R$ will be fixed.

\section{Existence, uniqueness and properties of the solution. \label%
{proof-existence}}

In this section, we give the proof of Theorem \ref{Existence}. It will go
along the following lines.

\begin{enumerate}
\item[(1)] Consider a sequence of processes $\left\{ \Xi _{t}^{n}\right\}
_{n\in \mathbb{N}}$, $t\in \mathcal{T}$, that solve finite volume cutoffs of system
(\ref{MainSystem}), and prove their uniform bound in $\mathcal{R}_{\beta
}^{p}$ for any $\beta >\alpha $. For this, we use our version of the
comparison theorem and Gronwall-type inequality in the scale of spaces,
which is in turn based on the Ovsjannikov method, see Appendix \ref%
{sec-est-sol}.

\item[(2)] The uniform bound above implies the convergence of sequence $\Xi
^{n}$, $n\rightarrow \infty $, to a process $\Xi =(\xi _{x})_{x\in \gamma
}\in \mathcal{R}_{\beta }^{p}$, $\beta >\alpha $. Our next goal is to prove
that the process $\Xi $ solves system (\ref{MainSystem}). The multiplicative
noise term does not allow to achieve this by a direct limit transition.
Therefore, we construct an $\mathbb{R}$-valued process $\eta _{t}$ that
solves an equation describing the dynamics of a tagged particle $x$, while
processes $\xi _{y,t}$, $y\in \gamma ,y\neq x$, are fixed, and prove that $
\eta _{t}=\xi _{x,t}$.

\item[(3)] The uniqueness and continuous dependence on the initial data is
proved by using our version of a Gronwall-type inequality, as above in part
(1). \GeorgyMarkup{The continuity of components of $\Xi$ will follow from our work on the dynamics of a tagged particle $x$ in section \ref{sec-1dspecialcase}.}
\end{enumerate}

Finally, in Subsection \ref{Markov}, we introduce Markov semigroup defined
by the solution of (\ref{MainSystem}).

\subsection{Truncated System \label{sec-Truncsys}}

Let us fix an expanding sequence $\{\Lambda _{n}\}_{n\in \mathbb{N}}$ of
finite subsets of $\gamma $ such that $\Lambda _{n}\uparrow \gamma $ as $%
n\rightarrow \infty $ and consider the following system of equations:%
\begin{eqnarray}
\xi _{x,t}^{n} &=&\zeta _{x}+\int_{0}^{t}\Phi _{x}(\Xi
_{s}^{n})ds+\int_{0}^{t}\Psi _{x}(\Xi _{s}^{n})dW_{x,s},\ x\in \Lambda _{n},
\label{FinVolSystem} \\
\xi _{x,t}^{n} &=&\zeta _{x},\quad x\not\in \Lambda _{n},\ \ t\in \mathcal{T},  \notag
\end{eqnarray}%
\GeorgyMarkup{where $\bar{\zeta}=\{\zeta_{x}\}_{x\in\gamma}\in L^{p}_{\alpha}$, $\alpha\in\mathcal{A}$, is $\mathcal{F}_{0}$-measurable random initial condition} 
\begin{color}{black} and equality (\ref{FinVolSystem}) holds  for all $t\in \mathcal{T}$, $\mathbb{P}$-a.s. \end{color}
Observe that for each $n\in \mathbb{N}$ system (\ref{FinVolSystem}) is a
truncated version of our original stochastic system (\ref{MainSystem}).

\begin{theorem}
\label{FiniteVolumeLemma} For any $n\in \mathbb{N}$ system (\ref%
{FinVolSystem}) \GeorgyMarkup{admits a \begin{color}{black} unique (up to indistinguishability) \end{color} solution $\Xi ^{n}\in \mathcal{R}_{\alpha }^{p}$ 
with continuous sample paths.}
\end{theorem}

\begin{proof}
 \GeorgyMarkup{The existence and \begin{color}{black} uniqueness \end{color} of continuous strong solutions of the non-trivial finite dimensional part of system (%
\ref{FinVolSystem}) is well-known, see \cite[Chapter 3]{LiuRockner}}. The inclusion $\Xi ^{n}\in \mathcal{R}_{\alpha }^{p}$
follows then from the fact that $\xi _{x,t}^{n}=\zeta _{x},\ t\in \mathcal{T}
$, for $x\not\in \Lambda _{n}$.
\end{proof}

\bigskip

Our next goal is to show that the sequence $\{\Xi ^{n}\}_{n\in \mathbb{N}}$
converges in $\mathcal{R}_{\beta }^{p}$ for any $\beta >\alpha $. We start
with the following uniform estimate, which is rather similar to the one from 
\cite{KKP}, adapted to the framework of the scale of Banach spaces using our
version of the Gronwall inequality.

\begin{theorem}
\label{TailTheorem1} Let $\Xi ^{n}=(\xi _{x}^{n})_{x\in \gamma }$, $n\in 
\mathbb{N}$, be the sequence of process defined by Theorem \ref%
{FiniteVolumeLemma}. Then for all $\beta >\alpha $ we have 
\begin{equation}
\sum_{x\in \gamma }e^{-\beta |x|}\sup_{n\in \mathbb{N}}\sup_{t\in \mathcal{T}%
}\mathbb{E}\bigg{[}|\xi _{x,t}^{n}|^{p}\bigg{]}<\infty .  \label{conv-ser}
\end{equation}
\end{theorem}

\begin{proof}
It follows from the first part of Lemma \ref{two-bound} in Appendix (with $%
\xi ^{1}\equiv \xi ^{n}$) that for all $x\in \Lambda _{n}$ and $t\in 
\mathcal{T}$ we have 
\begin{equation}
\mathbb{E}\bigg{[}|\xi _{x,t}^{n}|^{p}\bigg{]}\leq \GeorgyMarkup{\mathbb{E}}|\zeta
_{x}|^{p}+C_{1}n_{x}^{2}\sum_{y\in \bar{\gamma}_{x}}\int_{0}^{t}\mathbb{E}%
\bigg{[}|\xi _{y,s}^{n}|^{p}\bigg{]}ds+\GeorgyMarkup{C_{2}^{x}}.  \label{ineq1111}
\end{equation}%
We remark that inequality above trivially holds for $x\notin \Lambda _{n}$,
because in this case $\xi _{x,t}^{n}=\zeta _{x}$ and all terms in the
right-hand side of the inequality are non-negative.

We now define a measurable map $\GeorgyMarkup{\eta}^{n}:\mathcal{T}\rightarrow
\GeorgyMarkup{l_{\alpha}^{1}}$ via the following formula 
\begin{equation*}
\GeorgyMarkup{\eta} _{x}^{n}(t)%
\coloneqq%
\max_{m\leq n}\ \mathbb{E}\bigg{[}|\xi _{x,t}^{m}|^{p}\bigg{]},\quad \forall
(t\in \mathcal{T}).
\end{equation*}%
It is immediate that its components satisfy inequality similar to (\ref%
{ineq1111}), that is,%
\begin{equation*}
\GeorgyMarkup{\eta}_{x}^{n}(t)\leq \GeorgyMarkup{\mathbb{E}}|\zeta _{x}|^{p}+C_{1}n_{x}^{2}\sum_{y\in \bar{\gamma%
}_{x}}\int_{0}^{t}\GeorgyMarkup{\eta}_{y}^{n}(s)ds+\GeorgyMarkup{C_{2}^{x}}.
\end{equation*}%
Set $\GeorgyMarkup{\theta_{x}}= \GeorgyMarkup{\mathbb{E}}|\zeta _{x}|^{p}+\GeorgyMarkup{C_{2}^{x}}$ and observe that $(\GeorgyMarkup{\theta_{x}})_{x\in \gamma }\in l_{\alpha }^{1}$. Then the map $\GeorgyMarkup{\eta}^{n}$
fulfills the conditions of Lemma \ref{gron111} in Appendix, which implies
that for all $n\in \mathbb{N}$ and $\beta >\alpha $ we have 
\begin{equation*}
\sum_{x\in \gamma }e^{-\beta |x|}\sup_{t\in \mathcal{T}}\GeorgyMarkup{\eta}
_{x}^{n}(t)\leq \GeorgyMarkup{K_{T}}(\alpha ,\beta )\sum_{x\in \gamma }e^{-\alpha
|x|}\GeorgyMarkup{\theta_{x}}<\infty .
\end{equation*}%
Observe that the left-hand side forms an increasing sequence, which implies
that it converges and 
\begin{equation*}
\lim_{n\rightarrow \infty }\sum_{x\in \gamma }e^{-\beta |x|}\sup_{t\in 
\mathcal{T}}\GeorgyMarkup{\eta}_{x}^{n}(t)\leq \GeorgyMarkup{K_{T}}(\alpha ,\beta )\sum_{x\in \gamma
}e^{-\alpha |x|}\GeorgyMarkup{\theta_{x}}<\infty .
\end{equation*}%
Then, for any finite set $\eta \subset \gamma $, we have%
\begin{equation*}
\sum_{x\in \eta }e^{-\beta |x|}\lim_{n\rightarrow \infty
}\sup_{t\in \mathcal{T}}\GeorgyMarkup{\eta}_{x}^{n}(t) = \lim_{n\rightarrow
\infty }\sum_{x\in \eta }e^{-\beta |x|}\sup_{t\in \mathcal{T}%
}\GeorgyMarkup{\eta}_{x}^{n}(t) 
\leq \GeorgyMarkup{K_{T}}(\alpha ,\beta )\sum_{x\in \gamma }e^{-\alpha |x|}\GeorgyMarkup{\theta_{x}}.
\end{equation*}%
On the other hand, it is clear that 
\begin{equation*}
\lim_{n\rightarrow \infty }\GeorgyMarkup{\eta}_{x}^{n}(t)=\sup_{n\in \mathbb{N}%
}\max_{m\leq n}\sup_{t\in \mathcal{T}}\ \mathbb{E}\bigg{[}|\xi
_{x,t}^{m}|^{p}\bigg{]}=\GeorgyMarkup{\sup_{n\in \mathbb{N}}\sup_{t\in \mathcal{T}}\mathbb{%
E}\bigg{[}|\xi _{x,t}^{n}|^{p}\bigg{]}}
\end{equation*}%
for any $x\in \gamma .$ Thus%
\begin{equation*}
\sum_{x\in \eta }e^{-\beta |x|}\sup_{t\in \mathcal{T}}\sup_{n\in \mathbb{N}}%
\mathbb{E}\bigg{[}\GeorgyMarkup{|\xi _{x,t}^{n}|^{p}}\bigg{]}\leq \GeorgyMarkup{K_{T}}(\alpha ,\beta
)\sum_{x\in \gamma }e^{-\alpha |x|}\GeorgyMarkup{\theta_{x}}.
\end{equation*}%
The latter inequality holds for all finite $\eta \subset \gamma $, which
implies that 
\begin{equation*}
\sum_{x\in \gamma }e^{-\beta |x|}\sup_{t\in \mathcal{T}}\sup_{n\in \mathbb{N}%
}\mathbb{E}\bigg{[}|\xi _{x,t}^{n}|^{p}\bigg{]}\leq \GeorgyMarkup{K_{T}}(\alpha ,\beta
)\sum_{x\in \gamma }e^{-\alpha |x|}\GeorgyMarkup{\theta_{x}},
\end{equation*}%
and the proof is complete.
\end{proof}

\begin{theorem}
\label{CauchySequenceTheorem} The sequence $\{\Xi ^{n}\}_{n\in \mathbb{N}}$
is Cauchy in $\mathcal{R}_{\beta }^{p}$ for any $\beta >\alpha $.
\end{theorem}

\begin{proof}
Let us fix $n,m\in \mathbb{N}$ and and assume, without loss of generality,
that $\Lambda _{n}\subset \Lambda _{m}$. We first consider the situation
where $x\in \Lambda _{n}$. It follows from the second part of Lemma \ref%
{two-bound} in Appendix (with $\xi ^{(1)}\equiv \xi ^{n}$ and $\xi
^{(2)}\equiv \xi ^{m}$) that for all $x\in \Lambda _{n}$ and $t\in \mathcal{T%
}$ we have%
\begin{align}
\mathbb{E}|\bar{\xi}_{x,t}^{n,m}|^{p}& \leq Bn_{x}^{2}\sum_{y\in \bar{\gamma}%
_{x}}\int_{0}^{t}\mathbb{E}|\bar{\xi}_{y,s}^{n,m}|^{p}ds,
\label{Theorem5EqnA} \\
\bar{\xi}_{x,t}^{n,m}& =\xi _{x,t}^{n}-\xi _{x,t}^{m}.  \label{xibar}
\end{align}

In the case where $x\in \Lambda _{m}\setminus \Lambda _{n}$ we see that for
all $t\in \mathcal{T}$ 
\begin{equation*}
|\bar{\xi}_{x,t}^{n,m}|^{p}\leq (|\xi _{x,t}^{n}|+|\xi _{x,t}^{m}|)^{p}\leq
2^{p-1}|\xi _{x,t}^{n}|^{p}+2^{p-1}|\xi _{x,t}^{m}|^{p},
\end{equation*}%
so that 
\begin{equation}
\mathbb{E}\bigg{[}|\bar{\xi}_{x,t}^{n,m}|^{p}\bigg{]}\leq 2^{p}\sup_{n\in 
\mathbb{N}}\mathbb{E}\bigg{[}|\xi _{x,t}^{n}|^{p}\bigg{]}\leq
2^{p}1_{\Lambda _{m}\setminus \Lambda _{n}}(x)\sup_{n\in \mathbb{N}%
}\sup_{t\in \mathcal{T}}\mathbb{E}\bigg{[}|\xi _{x,t}^{n}|^{p}\bigg{]}<\infty
\label{Theorem5EqnB}
\end{equation}%
(cf. Theorem \ref{TailTheorem1}). Combining equations (\ref{Theorem5EqnA})
and (\ref{Theorem5EqnB}) and taking into account that $\bar{\xi}%
_{x,t}^{n,m}=0$ for $x\notin \Lambda _{m}$, we obtain the inequality 
\begin{equation*}
\mathbb{E}\bigg{[}|\bar{\xi}_{x,t}^{n,m}|^{p}\bigg{]}\leq
B_{1}n_{x}^{2}\sum_{y\in \bar{\gamma}_{x}}\int_{0}^{t}\mathbb{E}\bigg{[}|%
\bar{\xi}_{y,s}^{n,m}|^{p}\bigg{]}ds+2^{p}1_{\Lambda _{m}\setminus \Lambda
_{n}}(x)\sup_{n\in \mathbb{N}}\sup_{t\in \mathcal{T}}\mathbb{E}\bigg{[}|\xi
_{x,t}^{n}|^{p}\bigg{]}
\end{equation*}%
for all $x\in \gamma $ and $t\in \mathcal{T}$ . We can now proceed as in the
proof of Theorem \ref{TailTheorem1}. Define a measurable map $\varrho ^{n,m}:%
\mathcal{T}\rightarrow l_{\alpha }^{1}$ via the formula 
\begin{equation*}
\varrho _{x}^{n,m}(t)%
\coloneqq%
\mathbb{E}\bigg{[}|\bar{\xi}_{x,t}^{n,m}|^{p}\bigg{]},\quad t\in \mathcal{T},
\end{equation*}%
and set 
\begin{equation*}
b_{x}=2^{p}1_{\Lambda _{m}\setminus \Lambda _{n}}(x)\sup_{n\in \mathbb{N}%
}\sup_{t\in \mathcal{T}}\mathbb{E}\bigg{[}|\xi _{x,t}^{n}|^{p}\bigg{]}.
\end{equation*}%
Obviously, \GeorgyMarkup{$(b_{x})_{x\in \gamma}\in l_{\alpha ^{\prime }}^{1}$
for any fixed $\alpha ^{\prime }\in (\alpha ,\beta)$.} It \GeorgyMarkup{therefore}
follows then from Lemma \ref{gron111} in Appendix that 
\begin{equation*}
\sum_{x\in \gamma }e^{-\beta |x|}\sup_{t\in \mathcal{T}}\varrho
_{x}^{n,m}(t)\leq \GeorgyMarkup{K_{T}}(\alpha ^{\prime },\beta )\sum_{x\in \gamma }e^{-\alpha
^{\prime }|x|}b_{x}.
\end{equation*}%
So we have shown that the following inequality holds: 
\begin{multline}
\Vert \Xi ^{n}-\Xi ^{m}\Vert _{\mathcal{R}_{\beta }^{p}}^{p}\leq
2^{p}\GeorgyMarkup{K_{T}}(\alpha ^{\prime },\beta )\sum_{x\in \Lambda _{m}\setminus \Lambda
_{n}}e^{-\alpha ^{\prime }|x|}\sup_{n\in \mathbb{N}}\sup_{t\in \mathcal{T}}%
\mathbb{E}\bigg{[}|\xi _{x,t}^{n}|^{p}\bigg{]}  \label{CauchyEqn} \\
\leq 2^{p}\GeorgyMarkup{K_{T}}(\alpha ^{\prime },\beta )\sum_{x\in \gamma \setminus\Lambda
_{n}}e^{-\alpha ^{\prime }|x|}\sup_{n\in \mathbb{N}}\sup_{t\in \mathcal{T}}%
\mathbb{E}\bigg{[}|\xi _{x,t}^{n}|^{p}\bigg{]}.
\end{multline}%
It follows from Theorem \ref{TailTheorem1} that the right hand side of (\ref%
{CauchyEqn}) is the remainder of the convergent series (\ref{conv-ser})
(with $\alpha ^{\prime }$ in place of $\beta $), which completes the proof.
\end{proof}

\subsection{One Dimensional Special Case \label{sec-1dspecialcase}}

We have shown in the previous section that, for any $\beta >\alpha $, the
sequence $\{\Xi ^{n}\}_{n\in \mathbb{N}}$ is Cauchy in the Banach space $%
\mathcal{R}_{\beta}^{p}$ and thus converges in this space. So we are now
in a position to define the process 
\begin{equation}
\overbrace{\ \Xi 
\coloneqq%
\lim_{n\rightarrow \infty }\Xi ^{n}\ }^{\text{in}\ \mathcal{R}_{\beta }^{p}}.
\label{mainprocess}
\end{equation}%
This process is a candidate for a solution of the system (\ref{MainSystem}).
A standard way to show this would be to pass to the limit on both sides of (%
\ref{FinVolSystem}). This approach requires however somewhat stronger
convergence than that in $\mathcal{R}_{\beta }^{p}$. We are going to
overcome this difficulty by considering special one-dimensional equations.

Consider an arbitrary $x\in \gamma $. It is convenient to consider elements
of $S^{\gamma }$ as pairs $(\sigma _{x},Z^{(x)})$, where $\sigma _{x}\in S$
and $Z^{(x)}=\left( z_{y}\right) _{y\in \gamma \backslash x}\in S^{\gamma
\backslash x}$. In these notations, we can write $\Phi _{x}(\Xi _{s})=\Phi
_{x}(\xi _{x,s},\Xi _{s}^{(x)})$ and $\Psi _{x}(\Xi _{s})=\Psi _{x}(\xi
_{x,s},\Xi _{s}^{(x)})$, where 
\GeorgyMarkup{
\begin{align}
\Xi ^{(x)}:=\left( \xi _{y}\right) _{y\in\gamma \backslash x}.	\label{Xiwithnox}
\end{align}}

Let \GeorgyMarkup{us} now fix process $\Xi $ defined by (\ref{mainprocess}) and consider the
following one-dimensional equation: 
\begin{equation}
\eta _{x,t}=\zeta _{x}+\int_{0}^{t}\Phi _{x}(\eta _{x,s},\Xi
_{s}^{(x)})ds+\int_{0}^{t}\Psi _{x}(\eta _{x,s},\Xi _{s}^{(x)})dW_{x}(s), \label{maineqn}
\end{equation}%
for all $t\in \mathcal{T} $, $\mathbb{P}$-a.s.  The main goal of this section is to prove that the equation (\ref{maineqn})
has a unique  solution $\eta _{x,t}$.

\begin{remark}
Note that, for a fixed $x\in \gamma $, the principal difference between
equations (\ref{maineqn}) and (\ref{MainSystem}) is that the process $\Xi $
is fixed in (\ref{maineqn}) and defined by the limit (\ref{mainprocess}),
which makes (\ref{maineqn}) a one-dimensional equation w.r.t. $\eta _{x}$.
\end{remark}

\noindent In order to establish the existence of a solution of equation (%
\ref{maineqn}) we need the following auxiliary result.

\begin{theorem}
\label{TAB} Let $x\in \gamma $ and $\xi _{x}$ be an $x$-component of the
process $\Xi $ defined by (\ref{mainprocess}). Then \GeorgyMarkup{sample paths of $\xi_{x}$ are a.s. continuous and}
\begin{equation}
\mathbb{E}\bigg{[}\sup_{t\in \mathcal{T}}|\xi _{x,t}|^{p}\bigg{]}<\infty .
\label{Esup}
\end{equation}
\end{theorem}

\begin{proof}
\GeorgyMarkupx{It is sufficient to show that, for a fixed $x\in \gamma $, the sequence $%
\left\{ \xi _{x}^{n}\right\} _{n\in \mathbb{N}}$, is Cauchy in the norm $%
\left(\mathbb{E}\bigg{[}\sup_{t\in \mathcal{T}}|\cdot |^{p}\bigg{]}\right) ^{1/p}$ because then there exists a subsequence $\left\{ \xi _{x}^{n_{k}}\right\} _{k\in \mathbb{N}}$ such that 
\begin{align}
\lim_{k\to\infty}\sup_{t\in \mathcal{T}}|\xi_{x,t}^{n_{k}} - \xi_{x,t}| = 0,\  \mathbb{P}-a.s.,   \nonumber
\end{align}
which, \begin{color}{black} together with the path-continuity of processes $\xi^{n_{k}}_{x,t}$, \end{color} implies the statement of the theorem.}

Fix $\bar{N}\in \mathbb{N}$ such that $x\in \Lambda _{\bar{N}}$ and $n,m\geq 
\bar{N}$ and assume, without loss of generality, that $n<m$ so that $x\in
\Lambda _{n}\subset \Lambda _{m}$ Consider the process $\bar{\xi}%
_{x,t}^{n,m} $ defined in (\ref{xibar}) and proceed as in Appendix, Lemma %
\ref{two-bound}, with $\xi ^{(1)}\equiv \xi ^{n}$ and $\xi ^{(2)}\equiv \xi
^{m}$. Taking $\sup_{t\in \mathcal{T}}$ of both sides of the equality (\ref%
{form222}) we obtain the bound 
\begin{equation}
\mathbb{E}\bigg{[}\sup_{t\in \mathcal{T}}|\bar{\xi}_{x,t}^{n,m}|^{p}\bigg{]}%
\leq K+\mathbb{E}\bigg{[}\sup_{t\in \mathcal{T}}\int_{0}^{t}p(\bar{\xi}%
_{x,s}^{n,m})^{p-1}\Psi _{x}^{n,m}(s)dW_{x}(s)\bigg{]},  \label{AA5}
\end{equation}%
where 
\begin{equation}
K%
\coloneqq%
Bn_{x}^{2}\sum_{y\in \bar{\gamma}_{x}}\int_{0}^{T}\mathbb{E}\bigg{[}|\bar{\xi%
}_{y,s}^{n,m}|^{p}\bigg{]}ds\leq Bn_{x}^{2}T\GeorgyMarkup{\sum_{y\in \bar{\gamma}_{x}}}\sup_{t\in \mathcal{T}}\mathbb{E}%
\bigg{[}|\bar{\xi}_{y,t}^{n,m}|^{p}\bigg{]}  \label{K}
\end{equation}%
and%
\begin{equation*}
\Psi _{x}^{n,m}(s):=\GeorgyMarkup{\Psi _{x}(\Xi _{s}^{n})-\Psi _{x}(\Xi
_{s}^{m}).}
\end{equation*}%
Now using first the Burkholder-Davis-Gundy inequality \GeorgyMarkup{(see \cite{RocknerBDG})} and then the Jensen
inequality we see that the following estimate on the stochastic term from (%
\ref{AA5}) holds. 
\begin{multline}
\mathbb{E}\bigg{[}\sup_{t\in \mathcal{T}}\int_{0}^{t}p(\bar{\xi}%
_{x,s}^{n,m})^{p-1}\Psi _{x}^{n,m}(s)dW_{x}(s)\bigg{]}\leq \mathbb{E}\bigg{[}%
\bigg{(}\int_{0}^{t}\bigg{(}p(\bar{\xi}_{x,s}^{n,m})^{p-1}\Psi _{x}^{n,m}(s)%
\bigg{)}^{2}ds\bigg{)}^{\frac{1}{2}}\bigg{]} \\
\leq \bigg{(}\mathbb{E}\bigg{[}\int_{0}^{t}\bigg{(}p(\bar{\xi}%
_{x,s}^{n,m})^{p-1}\Psi _{x}^{n,m}(s)\bigg{)}^{2}ds\bigg{]}\bigg{)}^{\frac{1%
}{2}}.  \label{AA6}
\end{multline}%
The integrand in the right-hand side of the above inequality can be
estimated in a similar way as (\ref{ConditionCConsequence}), so that we
obtain 
\begin{equation*}
\bigg{(}(\bar{\xi}_{x,t}^{n,m})^{p-1}\Psi _{x}^{n,m}(t)\bigg{)}^{2}\leq
2M^{2}\GeorgyMarkup{(n_{x}+1)^{2}}|\bar{\xi}_{x,t}^{n,m}|^{2p}+2M^{2}\GeorgyMarkup{n_{x}^{2}}\sum_{y\in \bar{%
\gamma}_{x}}|\bar{\xi}_{y,t}^{n,m}|^{2p}.
\end{equation*}%
It follows now that inequality (\ref{AA6}) can be written in the following
way: 
\begin{multline*}
\mathbb{E}\bigg{[}\sup_{t\in \mathcal{T}}\int_{0}^{t}p(\bar{\xi}%
_{x,s}^{n,m})^{p-1}\Psi _{x}^{n,m}(s)dW_{x}(s)\bigg{]}\leq C_{1}\GeorgyMarkup{\sqrt{\sup_{t\in 
\mathcal{T}}\mathbb{E}\bigg{[}|\bar{\xi}_{x,t}^{n,m}|^{2p}\bigg{]}}} 
+C_{2}\GeorgyMarkup{\sqrt{\sum_{y\in \bar{\gamma}_{x}}\sup_{t\in \mathcal{T}}\mathbb{E}\bigg{[}|%
\bar{\xi}_{y,t}^{n,m}|^{2p}\bigg{]}}},
\end{multline*}%
where 
\begin{equation*}
C_{1}%
\coloneqq%
\GeorgyMarkup{\sqrt{2p^{2}M^{2}(1 + n_{x})^{2}T}}\quad \text{and}\quad C_{2}%
\coloneqq%
\GeorgyMarkup{\sqrt{2p^{2}M^{2}n_{x}^{2}T}}.
\end{equation*}%
Therefore returning to inequalities (\ref{AA5}) and (\ref{K}) we see that 
\begin{multline}
\mathbb{E}\bigg{[}\sup_{t\in \mathcal{T}}|\bar{\xi}_{x,t}^{n,m}|^{p}\bigg{]}%
\leq Bn_{x}^{2}T\GeorgyMarkup{\sum_{y\in \bar{\gamma}_{x}}}\sup_{t\in \mathcal{T}}\mathbb{E}\bigg{[}|\bar{\xi}%
_{y,t}^{n,m}|^{p}\bigg{]}  \label{AA7} \\
+C_{1}\GeorgyMarkup{\sqrt{\sup_{t\in \mathcal{T}}\mathbb{E}\bigg{[}|\bar{\xi}_{x,t}^{n,m}|^{2p}%
\bigg{]}}}+C_{2}\GeorgyMarkup{\sqrt{\sum_{y\in \bar{\gamma}_{x}}\sup_{t\in \mathcal{T}}\mathbb{E}%
\bigg{[}|\bar{\xi}_{y,t}^{n,m}|^{2p}\bigg{]}}}.
\end{multline}%
Since $\bar{\gamma}_{x}$ is finite we can now use Theorem \ref%
{CauchySequenceTheorem} to conclude that, with a suitable choice of $n,m\in 
\mathbb{N}$, the right hand side of the inequality (\ref{AA7}) above can be
made arbitrary small hence the proof is complete.
\end{proof}

\begin{theorem}
\label{globaltheorem} Equation (\ref{maineqn}) admits a unique 
solution.
\end{theorem}

\begin{proof}
By standard arguments, see e.g. \cite[Proposition 2.9]{ABW}, we conclude
that equation (\ref{maineqn}) admits a unique local maximal solution $\eta
_{x}$ such that \begin{color}{black}
\begin{equation*}
\eta _{x,t\wedge \tau _{n}}=\zeta _{x}+\int_{0}^{t\wedge \tau _{n}}\Phi
_{x}(\eta _{x,s\wedge \tau _{n}},\Xi _{s\wedge \tau
_{n}}^{(x)})ds+\int_{0}^{t\wedge \tau _{n}}\Psi _{x}(\eta _{x,s\wedge \tau
_{n}},\Xi _{s\wedge \tau _{n}}^{(x)})dW_{x}(s),
\end{equation*}%
for all $t\in \mathcal{T} $, $\mathbb{P}$-a.s.  \end{color}  Here \GeorgyMarkup{$\Xi^{(x)}_{s\land\tau_{n}}$ is as in (\ref{Xiwithnox})} and by construction, for all $n\in \mathbb{N}$, stopping time $\tau _{n}$
is the first exit time of $\eta _{x}$ from the interval $(-n,\ n)$, defined
as%
\begin{equation*}
\tau _{n}=\left\{ 
\begin{array}{c}
T,\text{ if }\left\vert \GeorgyMarkup{\eta_{x,t}}\right\vert <n,~t\in \lbrack 0,T] \\ 
\inf \left\{ t\in \lbrack 0,T]:\left\vert \GeorgyMarkup{\eta _{x,t}}\right\vert \geq
n\right\} ,\ \text{otherwise}%
\end{array}%
\right.
\end{equation*}
Hence to complete the proof it is sufficient to establish that almost surely 
$\lim_{n\rightarrow \infty }\tau _{n}=T$. We will prove this fact along the
lines of \cite[Theorem 3.1]{ABW}, using the bound (\ref{Esup}). We begin by
using the Itô Lemma to establish the equality 
\GeorgyMarkupy{
\begin{multline*}
|\eta _{x,t\wedge \tau _{n}}|^{p}=\int_{0}^{t\wedge \tau _{n}}p(\eta
_{x,s\wedge \tau _{n}})^{p-1}\Phi _{x}(\eta _{x,s\wedge \tau _{n}},\Xi
_{s\wedge \tau _{n}}^{(x)})ds+ \\
+\int_{0}^{t\wedge \tau _{n}}\frac{p(p-1)}{2}(\eta _{x,s\wedge \tau
_{n}})^{p-2}(\Psi _{x}(\eta _{x,s\wedge \tau _{n}},\Xi _{s\wedge \tau
_{n}}^{(x)}))^{2}ds+ \\
+\int_{0}^{t\wedge \tau _{n}}p(\eta _{x,s\wedge \tau _{n}})^{p-1}\Psi
_{x}(\eta _{x,s\wedge \tau _{n}},\Xi _{s\wedge \tau _{n}}^{(x)})dW_{x}(s),
\end{multline*}}
for all $t\in \mathcal{T}$. Before proceeding we define for convenience the
following shorthand notations: 
\begin{align*}
\bar{\Phi}_{x}^{p}(\eta ,t)& :=(\eta _{x,t\wedge \tau _{n}})^{p-1}\Phi
_{x}(\eta _{x,t\wedge \tau _{n}},\Xi _{t\wedge \tau _{n}}^{(x)}), \\[0.01in]
\bar{\Psi}_{x}^{p}(\eta ,t)& :=(\eta _{x,t\wedge \tau _{n}})^{p-2}(\Psi
_{x}(\eta _{x,t\wedge \tau _{n}},\Xi _{t\wedge \tau _{n}}^{(x)}))^{2}, \\
\GeorgyMarkup{w_{x}}&\GeorgyMarkup{\coloneqq b+\frac{1}{2} + 4\bar{a}^{2}n_{x}^{2}}, \\
\GeorgyMarkup{u_{x}}&\GeorgyMarkup{\coloneqq c+\bar{a}n_{x}.}
\end{align*}%
An application of Lemma \ref{DriftLemma} in the Appendix shows that for all $%
t\in \mathcal{T}$ we have 
\GeorgyMarkupy{
\begin{multline}
\bar{\Phi}_{x}^{p}(\eta ,t)\leq \left\vert \eta _{x,t\wedge \tau
_{n}}\right\vert ^{p-2}  \label{BB2} \left( \GeorgyMarkup{w_{x}}|\eta _{x,t\wedge \tau _{n}}|^{2}+\frac{1}{2}%
\bar{a}^{2}n_{x}\sum_{y\in \gamma _{x}}|\xi _{y,t}|^{2}+\left\vert \eta
_{x,t\wedge \tau _{n}}\right\vert \GeorgyMarkup{u_{x}} \right) \\
\leq \GeorgyMarkup{w_{x}}|\eta _{x,t\wedge \tau _{n}}|^{p}+\frac{1}{2}\bar{a}%
^{2}n_{x}(\eta _{x,t\wedge \tau _{n}})^{p-2}\sum_{y\in \gamma _{x}}|\xi
_{y,t\wedge \tau _{n}}|^{2} + \left\vert \eta _{x,t\wedge \tau
_{n}}\right\vert ^{p-1}\GeorgyMarkup{u_{x}} \\
\leq (\GeorgyMarkup{w_{x}+ 2^{p-1}u_{x}})|\eta _{x,t\wedge \tau _{n}}|^{p}+\frac{1}{2}%
\bar{a}^{2}n_{x}\left\vert \eta _{x,t\wedge \tau _{n}}\right\vert
^{p-2}\sum_{y\in \gamma _{x}}|\xi _{y,t\wedge \tau _{n}}|^{2} + 2^{p-1}\GeorgyMarkup{u_{x}},
\end{multline}}
where constants $b$ and $c$ are defined in Assumption \ref{mainass}. In the
last inequality we used the simple estimate $C^{p-1}\leq (1+C)^{p-1}\leq
(1+C)^{p}\leq 2^{p-1}(1+C^{p})$ for any $C>0$, which holds because $p>1$. We
can now use the H\"{o}lder inequality and classical estimate $\left(
\sum_{k=1}^{m}a_{k}\right) ^{N}\leq m^{N-1}\sum_{k=1}^{m}a_{k}^{N}$ (see
e.g. \cite{Jameson}) in conjunction with inequality (\ref{BB2}) above to see
that for all $t\in \mathcal{T}$ we have 
\begin{multline*}
\mathbb{E}\bigg{[}\bar{\Phi}_{x}^{p}(\eta ,t)\bigg{]}\leq (\GeorgyMarkup{w_{x}+
2^{p-1}u_{x}})\mathbb{E}\bigg{[}|\eta _{x,t\wedge \tau _{n}}|^{p}\bigg{]} \\
+\frac{1}{2}\bar{a}^{2}n_{x}\bigg{(}\mathbb{E}\bigg{[}|\eta _{x,t\wedge \tau
_{n}}|^{p}\bigg{]}\bigg{)}^{\frac{p-2}{p}}\bigg{(}\mathbb{E}\bigg{[}\bigg{(}%
\sum_{y\in \gamma _{x}}|\xi _{y,t\wedge \tau _{n}}|^{2}\bigg{)}^{\frac{p}{2}}%
\bigg{]}\bigg{)}^{\frac{2}{p}}+2^{p-1}\GeorgyMarkup{u_{x}} \\
\leq (\GeorgyMarkup{w_{x}+2^{p-1}u_{x}})\mathbb{E}\bigg{[}|\eta _{x,t\wedge \tau
_{n}}|^{p}\bigg{]} \\
+\frac{1}{2}\bar{a}^{2}n_{x}\bigg{(}1+\mathbb{E}\bigg{[}|\eta _{x,t\wedge
\tau _{n}}|^{p}\bigg{]}\bigg{)}n_{x}^{\frac{p-2}{p}}\bigg{(}\mathbb{E}%
\bigg{[}\sum_{y\in \gamma _{x}}|\xi _{y,t\wedge \tau _{n}}|^{p}\bigg{]}%
\bigg{)}^{\frac{2}{p}}+2^{p-1}\GeorgyMarkup{u_{x}}.
\end{multline*}%
In a similar way, we obtain the inequality%
\begin{multline*}
\bar{\Psi}_{x}^{p}(\eta ,t)\leq \GeorgyMarkup{3(M^{2}(n_{x}+1)^{2}+M^{2}n_{x}^{2}2^{p-1})}|\eta _{x,t\wedge
\tau _{n}}|^{p} \\
+3M^{2}n_{x}^{2}(\eta _{x,t\wedge \tau _{n}})^{p-2}\sum_{y\in \gamma
_{x}}|\xi _{y,t\wedge \tau _{n}}|^{2}+\GeorgyMarkup{3M^{2}n_{x}^{2}2^{p-1}}.
\end{multline*}%
Setting 
\begin{equation*}
A_{x}:=\max \bigg{\{}\frac{1}{2}\bar{a}^{2}n_{x}^{1+\frac{p-2}{p}},\
3M^{2}n_{x}^{2}\bigg{\}}\bigg{(}\sum_{y\in \gamma _{x}}\mathbb{E}\bigg{[}%
\sup_{t\in \mathcal{T}}|\xi _{y,t}|^{p}\bigg{]}\bigg{)}^{\frac{2}{p}},
\end{equation*}%
we get the bounds%
\begin{equation*}
\mathbb{E}\bigg{[}\bar{\Phi}_{x}^{p}(\eta ,t)\bigg{]}\leq (\GeorgyMarkup{w_{x}+2^{p-1}u_{x}}+A_{x})\mathbb{E}\bigg{[}|\eta _{x,t\wedge \tau _{n}}|^{p}\bigg{]}%
+A_{x}+2^{p-1}\GeorgyMarkup{u_{x}}
\end{equation*}%
and 
\begin{equation*}
\mathbb{E}\bigg{[}\bar{\Psi}_{x}^{p}(\eta ,t)\bigg{]}\leq
(\GeorgyMarkup{3M^{2}(n_{x}+1)^{2}+3M^{2}n_{x}^{2}2^{p-1})}+A_{x})\mathbb{E}\bigg{[}|\eta _{x,t\wedge
\tau _{n}}|^{p}\bigg{]}+A_{x}+\GeorgyMarkup{3M^{2}n_{x}^{2}2^{p-1}}.
\end{equation*}%
Observe that $A_{x}<\infty $ by Theorem \ref{TAB}. Finally letting 
\begin{align*}
D& :=p(\GeorgyMarkup{w_{x}+2^{p-1}u_{x}+A_{x}})+\frac{p(p-2)}{2}%
(\GeorgyMarkup{3M^{2}(n_{x}+1)^{2}+3M^{2}n_{x}^{2}2^{p-1}+A_{x}}), \\
K& :=pT(\GeorgyMarkup{A_{x}+2^{p-1}u_{x}})+\frac{p(p-2)}{2}T(\GeorgyMarkup{A_{x}+3M^{2}n_{x}^{2}2^{p-1})},
\end{align*}%
we see that for all $t\in \lbrack 0,\infty )$ we have 
\begin{equation}
\mathbb{E}\bigg{[}|\eta _{x,t\wedge \tau _{n}}|^{p}\bigg{]}\leq D\int_{0}^{t}%
\mathbb{E}\bigg{[}|\eta _{x,s\wedge \tau _{n}}|^{p}\bigg{]}ds+K.  \label{BB4}
\end{equation}%
Observe that constants $K$ and $D$ are independent of the stopping time $%
\tau _{n}$.

The rest of the proof is standard and can be completed along the lines of 
\cite[Theorem 3.1]{ABW}. We give its sketch for the convenience of the reader.
Using Gronwall's inequality together with the inequality (\ref{BB4}) above
we see that for all $t\in \lbrack 0,T]$ we have 
\begin{equation*}
\mathbb{E}\bigg{[}|\eta _{x,t\wedge \tau _{n}}|^{p}\bigg{]}\leq Ke^{Dt}.
\end{equation*}%
It follows from the definition of stopping time $\tau _{n}$ that 
\begin{equation*}
\mathbb{E}\bigg{[}|\eta _{x,t\wedge \tau _{n}}|^{p}\bigg{]}\geq n^{p}\mathbb{%
P}(\tau _{n}<t),\ 
\end{equation*}%
so that, for all $t\in \lbrack 0,T]$,%
\begin{equation*}
\mathbb{P}(\tau _{n}<t)\leq \frac{1}{n^{p}}Ke^{Dt}\rightarrow 0,\
n\rightarrow \infty .
\end{equation*}%
Now convergence in probability and the fact that $\{\tau _{n}\}_{n\in 
\mathbb{N}}$ is an increasing sequence imply that almost surely $%
\lim_{n\rightarrow \infty }\tau _{n}=T$, hence the proof is complete.
\end{proof}

\subsection{Proof of Existence and Uniqueness \label{sec-eu}}

In this section, we are going to prove Theorem \ref{Existence}. We will show
that, for any$\ \beta >\alpha ,$ the process%
\begin{equation}
\overbrace{\ \Xi 
\coloneqq%
\lim_{n\rightarrow \infty }\Xi ^{n}\ }^{\text{in}\ \mathcal{R}_{\beta }^{p}}
\label{convY}
\end{equation}%
solves system (\ref{MainSystem}). For this, we will use auxiliary processes $%
\eta _{x}$ constructed in Theorem \ref{globaltheorem}.

\begin{proof}[Proof of the existence]
According to Theorem \ref{globaltheorem}, for each $x\in \gamma $ equation 
\begin{color}{black}
\begin{equation*}
\eta _{x,t}=\zeta _{x}+\int_{0}^{t}\Phi _{x}(\eta _{x,s},\GeorgyMarkup{\Xi
_{s}^{(x)})}ds+\int_{0}^{t}\Psi _{x}(\eta _{x,s},\GeorgyMarkup{\Xi
_{s}^{(x)})}dW_{x,s},  \, \text{for all} \; t\in \mathcal{T},\, \mathbb{P}-a.s.
\end{equation*}%
 where \GeorgyMarkup{$\Xi^{(x)}_{s}$ is as in (\ref{Xiwithnox})}, admits a unique solution $\eta _{x,t}$. Thus it is sufficient to prove that this solution
is indistinguishable from the process $\xi _{x}$. The convergence (\ref{convY}) implies
that, for any fixed $x\in \gamma $,%
\begin{equation}
\lim_{n\rightarrow \infty }\mathbb{E}|\xi _{x,t}^{n}-\xi _{x,t}|^{p}=0,\
t\in \mathcal{T}.  \label{ExEq2}
\end{equation}%
Therefore, taking into account that both processes $\xi_x$ and $\eta_x$ are continuous, \end{color} to conclude this proof it remains to show that, for any $%
\ t\in \mathcal{T}$, 
\begin{equation}
\lim_{n\rightarrow \infty }\mathbb{E}|\xi _{x,t}^{n}-\eta _{x,t}|^{p}=0.
\label{ExEq3}
\end{equation}%
Let us fix $x\in \gamma $ and$\ t\in \mathcal{T}$ and assume without loss of
generality that $x\in \Lambda _{n}\subset \gamma $. Define the following
processes: 
\begin{align*}
\Phi _{x}^{n}(t)\!& :=\Phi _{x}(\xi _{x,t}^{n},\Xi _{t}^{n})-\Phi _{x}(\eta
_{x,t},\GeorgyMarkup{\Xi _{t}^{(x)}}), \\
\Psi _{x}^{n}(t)\!& :=\Psi _{x}(\xi _{x,t}^{n},\Xi _{t}^{n})-\Psi _{x}(\eta
_{x,t},\GeorgyMarkup{\Xi _{t}^{(x)}}),\  \\
\mathcal{X}_{x,t}^{n}& :=\xi _{x,t}^{n}-\eta _{x,t}.
\end{align*}%
The rest of the proof is rather similar to the proof of Theorem \ref%
{globaltheorem}. The Itô Lemma shows that for all $t\in \mathcal{T}$ we have  $\mathbb{P}$-a.s. 
\begin{multline}
|\mathcal{X}_{x,t}^{n}|^{p}=\int_{0}^{t}p(\mathcal{X}_{x,s}^{n})^{p-1}\Phi
_{x}^{n}(s)ds+ \\
+\int_{0}^{t}\frac{p(p-1)}{2}(\mathcal{X}_{x,s}^{n})^{p-2}(\Psi
_{x}^{n}(s))^{2}ds+ \\
+\int_{0}^{t}p(\mathcal{X}_{x,s}^{n})^{p-1}\Psi _{x}^{n}(s)dW_{x}(s).
\label{ExEq4}
\end{multline}%
Using Lemma \ref{two-bound} in Appendix, we can see that for all $t\in 
\mathcal{T}$ 
\begin{equation*}
(\mathcal{X}_{x,t}^{n})^{p-1}\Phi _{x}^{n}(t) \leq (b+\frac{1}{2}%
\GeorgyMarkup{+4\bar{a}^{2}n_{x}^{2}})\left\vert \mathcal{X}_{x,t}^{n}\right\vert ^{p} 
 +\bar{a}^{2}n_{x}\left\vert \mathcal{X}_{x,t}^{n}\right\vert
^{p-2}\sum_{y\in \gamma _{x}}(\xi _{y,t}^{n}-\xi _{y,t})^{2},
\end{equation*}%
and 
\begin{align*}
(\mathcal{X}_{x,t}^{n})^{p-2}\Psi _{x}^{n}(t)^{2} \leq
2M^{2}\GeorgyMarkup{(n_{x}+1)^{2}}\left\vert \mathcal{X}_{x,t}^{n}\right\vert ^{p} 
 +2M^{2}n_{x}\left\vert \mathcal{X}_{x,t}^{n}\right\vert ^{p-2}\sum_{y\in
\gamma _{x}}(\xi _{y,t}^{n}-\xi _{y,t})^{2}.
\end{align*}%
As \GeorgyMarkup{in} the proof of Theorem \ref{globaltheorem}, we see that for all $t\in 
\mathcal{T}$ 
\begin{equation}
\mathbb{E}\bigg{[}(\mathcal{X}_{x,t}^{n})^{p-1}\Phi _{x}^{n}(t)\bigg{]}\leq
(b+\frac{1}{2}\GeorgyMarkup{+4\bar{a}^{2}n_{x}^{2}} + A_{x}^{n})\mathbb{E}\bigg{[}\left\vert \mathcal{X}%
_{x,t}^{n}\right\vert ^{p}\bigg{]}+A_{x}^{n}  \label{ExEq7}
\end{equation}%
and 
\begin{equation}
\mathbb{E}\bigg{[}(\mathcal{X}_{x,t}^{n})^{p-2}\bigg{(}\Psi _{x}^{n}(t)%
\bigg{)}^{2}\bigg{]}\leq (2M^{2}\GeorgyMarkup{(n_{x}+1)^{2}}+A_{x}^{n})\mathbb{E}\bigg{[}%
\left\vert \mathcal{X}_{x,t}^{n}\right\vert ^{p}\bigg{]}+A_{x}^{n},
\label{ExEq8}
\end{equation}%
where 
\begin{equation*}
A_{x}^{n}:=\max \bigg{\{}\GeorgyMarkup{\bar{a}^{2}n_{x},\
2M^{2}n_{x}}\bigg{\}}\GeorgyMarkup{\mathbb{E}\bigg{[}\sum_{y\in \gamma%
_{x}}(\xi _{y,t}^{n}-\xi _{y,t})^{2}\bigg{]}}.
\end{equation*}%
Now, because $\GeorgyMarkup{\gamma_{x}}$ is finite \GeorgyMarkup{and $p\geq 2$} it is clear from equation (\ref{ExEq2}) that 
\begin{equation*}
\GeorgyMarkup{\mathbb{E}\sum_{y\in \gamma_{x}}(\xi _{y,t}^{n}-\xi _{y,t})^{2}} \rightarrow 0,\ n\rightarrow \infty ,
\end{equation*}%
so \GeorgyMarkup{we see that} $A_{x}^{n}\rightarrow 0$ as $n\rightarrow \infty $. Therefore
using inequality (\ref{ExEq7}) and (\ref{ExEq8}) above we can conclude from
equation (\ref{ExEq4}) that for all $x\in \gamma $ and all $t\in \mathcal{T}$
we have 
\begin{equation*}
\mathbb{E}\bigg{[}|\mathcal{X}_{x,t}^{n}|^{p}\bigg{]}\leq
C_{x}^{n}\int_{0}^{t}\mathbb{E}\bigg{[}|\mathcal{X}_{x,s}^{n}|^{p}\bigg{]}ds+%
\bar{A}_{x}^{n},
\end{equation*}%
where 
\GeorgyMarkup{
\begin{align}
C_{x}^{n}&:=p(b+\frac{1}{2}+4\bar{a}^{2}n_{x}^{2}+A_{x}^{n})+\frac{p(p-1)}{2}
(2M^{2}(n_{x}+1)^{2}+A_{x}^{n}) \nonumber \\ 
\bar{A}_{x}^{n}&:=pTA_{x}^{n}+\frac{p(p-1)}{2}TA_{x}^{n}, \nonumber 
\end{align}
}
and consequently $C_{x}^{n},\bar{A}_{x}^{n}\rightarrow 0$ on $\mathcal{T}$
as $n\rightarrow \infty $. Finally using Gronwall inequality we see that for
all $t\in \mathcal{T}$ we have 
\begin{equation*}
\mathbb{E}\bigg{[}|\mathcal{X}_{x,t}^{n}|^{p}\bigg{]}\leq \bar{A}%
_{x}^{n}e^{(C_{x}^{n})T},
\end{equation*}%
which shows that for all $x\in \gamma $ and uniformly on $\mathcal{T}$ 
\begin{equation*}
\lim_{n\rightarrow \infty }\mathbb{E}\bigg{[}|\mathcal{X}_{x,t}^{n}|^{p}%
\bigg{]}=0.
\end{equation*}%
Equation (\ref{ExEq3}) now follows immediately hence the proof is complete.
\end{proof}

\bigskip

\begin{proof}[Proof of the uniqueness and continuous dependence]
Suppose that $\Xi _{t}^{1}=\left( \xi _{x,t}^{1}\right) _{x\in \gamma }$ and 
$\Xi _{t}^{2}=\left( \xi _{x,t}^{2}\right) _{x\in \gamma } \in \mathcal{R}_{\alpha +}^{p}$, are two  solutions of system (\ref%
{MainSystem}), with initial values \GeorgyMarkup{$\Xi _{0}^{1}, \Xi _{0}^{2}\in L^{p}_{\alpha}$},
respectively. \GeorgyMarkup{Now, for all $t\in\timeint$ and all $x\in\gamma$ letting $\bar{\xi}_{x,t} = \xi^{(1)}_{x,t} - \xi^{(2)}_{x,t}$ we see from} Lemma \ref{two-bound} that%
\begin{equation*}
\mathbb{E}\bigg{[}|\bar{\xi}_{x,t}|^{p}\bigg{]}\leq \GeorgyMarkup{\mathbb{E}}\left\vert \bar{\xi}%
_{x,0}\right\vert ^{p}+Bn_{x}^{2}\sum_{y\in \bar{\gamma}_{x}}\int_{0}^{t}%
\mathbb{E}\bigg{[}|\bar{\xi}_{y,s}|^{p}\bigg{]}ds.
\end{equation*}%
Fix an arbitrary $\beta >\alpha $ and $\alpha _{1}\in \left( \alpha ,\beta
\right) $. An application of Lemma \ref{gron111} to a bounded measurable map $\kappa
:\mathcal{T}\rightarrow l_{\alpha _{1}}^{1}$ defined by the formula 
\begin{equation*}
\kappa _{x}(t)%
\coloneqq%
\mathbb{E}\bigg{[}|\bar{\xi}_{x,t}|^{p}\bigg{]}.
\end{equation*}%
shows that 
\begin{equation*}
\sum_{x\in \gamma }e^{-\beta |x|}\sup_{t\in \mathcal{T}}\kappa _{x}(t)\leq
\GeorgyMarkup{K_{T}}(\alpha _{1},\beta )\GeorgyMarkup{\sum_{x\in \gamma }e^{-\alpha_{1} |x|}|b_{x}|},\ \beta
>\alpha _{1},
\end{equation*}%
where $b_{x}=\GeorgyMarkup{\mathbb{E}}|\bar{\xi}_{x,0}|^{p}$. Therefore we establish that 
\begin{equation*}
||\Xi ^{1}-\Xi ^{2}||_{\mathcal{R}_{\beta }^{p}}^{p}\ \equiv \ \sup_{t\in 
\mathcal{T}}\mathbb{E}\bigg{[}\sum_{x\in \gamma }e^{-\beta |x|}|\bar{\xi}%
_{x,t}|^{p}\bigg{]}\leq \GeorgyMarkup{K_{T}}(\alpha _{1},\beta )\sum_{x\in \gamma }e^{-\GeorgyMarkup{\alpha_{1}}
|x|}\GeorgyMarkup{\mathbb{E}}|\bar{\xi}_{x,0}|^{p},
\end{equation*}%
which implies both statements.
\end{proof}

\subsection{Markov semigroup \label{Markov}}
\AlexMarkup{
In this section we denote by $\Xi _{t}(\bar{\zeta})$ the solution of
equation (\ref{MainSystem}) with initial condition $\bar{\zeta}$. {\rm T}his
process generates an operator family ${\rm {\rm T}}_{t}:C_{b}(l_{\beta }^{p})\rightarrow
C_{b}(l_{\alpha }^{p})$, $\alpha <\beta $, $t\geq 0$, by standard formula%
\begin{equation}
	{\rm T}_{t}f(\bar{\zeta})=\mathbb{E}f(\Xi _{t}(\bar{\zeta})).  \label{sem-def}
\end{equation}%
Consider the space $l_{\alpha +}^{p}=\cap _{\beta >\alpha }l_{\beta }^{p}$
equipped with the projective limit topology, which makes it a Polish space 
see e.g. \cite{GelShi}.

\begin{theorem}
	Operator family ${\rm T}_{t},$ $t\geq 0$, is a strongly continuous Markov
	semigroup in $C_{b}(l_{\alpha +}^{p})$ for any $\alpha \in \mathcal{A}$.
\end{theorem}

\begin{proof}
	Continuity of the map $L_{\alpha }^{p}\ni \bar{\zeta}\mapsto \Xi (\bar{\zeta}%
	)\in \mathcal{R}_{\beta }^{p}$, $\alpha <\beta $, \begin{color}{black} for an arbitrary  $T>0$ \end{color} (cf. Theorem %
	\ref{Existence}), implies that operators ${\rm T}_{t}:C_{b}(l_{\beta
	}^{p})\rightarrow C_{b}(l_{\alpha }^{p})$, \begin{color}{black}  $t\ge 0$, \end{color}  are bounded, which in turn implies
	their boundedness as operators in $C_{b}(l_{\alpha +}^{p}),$ for any $\alpha
	\in \mathcal{A}$. {\rm T}he uniqueness of the solution (cf. Theorem \ref{Existence}%
	) implies in the standard way the evolution property
	\[
	{\rm T}_{t}{\rm T}_{s}={\rm T}_{t+s},\ t,s\geq 0.
	\]%
	Observe that the truncated process $\Xi _{t}^{n}(\bar{\zeta})$ generates the
	strongly continuous semigroup ${\rm T}_{t}^{n}:C_{b}(l_{\alpha }^{p})\rightarrow
	C_{b}(l_{\alpha }^{p})$, for any $\alpha \in \mathcal{A}$. It follows from
	the convergence%
	\[
	\Xi _{t}^{n}(\bar{\zeta})\rightarrow \Xi _{t}(\bar{\zeta}),\ n\rightarrow
	\infty ,
	\]%
	in $\mathcal{R}_{\beta }^{p}$ for any $\beta >\alpha $ that%
	\[
	\sup_{t\in \mathcal{T}}\left\Vert {\rm T}_{t}^{n}f(\bar{\zeta})-{\rm T}_{t}f(\bar{\zeta}%
	)\right\Vert _{C_{b}(l_{\alpha +}^{p})},\ n\rightarrow \infty ,\text{ for
		any }f\in C_{b}(l_{\alpha +}^{p})\text{ and }\bar{\zeta}\in l_{\alpha +}^{p},
	\]%
	which in turn implies that ${\rm T}_{t}:C_{b}(l_{\alpha +}^{p})\rightarrow
	C_{b}(l_{\alpha +}^{p})$ is strongly continuous. 
\end{proof}

\begin{remark}
	The dominated convergence theorem implies that 
	\begin{equation}
		\int {\rm T}_{t}^{n}f(\bar{\zeta})\nu (d\bar{\zeta})\rightarrow \int {\rm T}_tf(\bar{\zeta}%
		)\nu (d\bar{\zeta}),\ n\rightarrow \infty ,  \label{lim-sem}
	\end{equation}%
	for any probability measure $\nu $ on $l_{\alpha +}^{p}$.
\end{remark}}

\section{Stochastic dynamics associated with Gibbs measures \label{sdyn}}

As an application of our results, we will present a construction of
stochastic dynamics associated with Gibbs measures on $S^{\gamma }$. 
\begin{color}{black}%
Sufficient%
\end{color}
conditions of the existence of these measures were derived in \cite{DKKP}.
For the convenience of the reader, we start with a reminder of the general
definition of Gibbs measures, adopted to our framework.

\subsection{Construction of Gibbs measures \label{gibbs0}}

In the standard Dobrushin-Lanford-Ruelle\ (DLR) approach in statistical
mechanics \cite{Geor,Pre}, Gibbs measures (states) are constructed by means
of their local conditional distributions (constituting the so-called
Gibbsian specification). We are interested in Gibbs measures describing
equilibrium states of a (quenched) system of particles with positions $%
\gamma \subset X=\mathbb{R}^{d}$ and spin space $S=\mathbb{R}$, defined by
pair and single-particle potentials $W_{xy}$ and $V$, respectively. We
assume the following:

\begin{itemize}
\item $W_{xy}:S\times S\rightarrow \mathbb{R}$, $x,y\in X$, are measurable
functions satisfying the polynomial growth estimate 
\begin{equation}
\left\vert W_{xy}(u,v)\right\vert \leq I_{W}\left( \left\vert u\right\vert
^{r}+\left\vert v\right\vert ^{r}\right) +J_{W},\ \ \ u,v\in S,
\label{W-est}
\end{equation}%
and the finite range condition $W_{xy}\equiv 0$ if $\left\vert
x-y\right\vert \leq \rho $ for all $x,y\in X$ and some constants $%
I_{W},J_{W},R,r\geq 0$. We assume also that $W_{xy}(u,v)$ is symmetric with
respect to the permutation of $(x,u)$ and $(y,v)$.

\item the single-particle potential $V$ satisfies the bound 
\begin{equation}
V(u)\geq a_{V}\left\vert u\right\vert ^{\tau }-b_{V},\ \ u\in S,
\label{V-est}
\end{equation}%
for some constants $a_{V},b_{V}>0$, and $\tau >r$.
\end{itemize}

\begin{example}
\label{ex-gibbs}A typical example is given by the pair interaction in of the
form 
\begin{equation*}
W_{xy}(u,v)=a(x-y)u~v,\ \text{\ }u,v\in S,
\end{equation*}%
where $a:X\rightarrow \mathbb{R}$ is as in Section \ref{sec-stochsystem}. In
this case, $r=2$ and so we need $\tau >2$ in (\ref{V-est}). The method of 
\cite{DKKP} does not allow us to control the case of $\tau =2,$ even when
the underlying particle configuration $\gamma $ is a typical realisation of
a homogeneous Poisson random field on $\Gamma (X)$.
\end{example}

Let $\mathcal{F}(\gamma )$ be the collection of all finite subsets of $%
\gamma \in \Gamma (X)$. For any $\eta \in \mathcal{F}(\gamma ),$ $\bar{\sigma%
}_{\eta }=(\sigma _{x})_{x\in \eta }\in S^{\eta }$ and $\bar{z}_{\gamma
}=(z_{x})_{x\in \gamma }\in S^{\gamma }$ define the relative local
interaction energy%
\begin{equation*}
E_{\eta }(\bar{\sigma}_{\eta }\left\vert \bar{z}_{\gamma }\right.
)=\sum_{\{x,y\}\subset \eta }W_{xy}(\sigma _{x},\sigma _{y})+\sum_{\substack{
x\in \eta  \\ y\in \gamma \setminus \eta }}W_{xy}(\sigma _{x},z_{y}).
\end{equation*}%
The corresponding specification kernel $\Pi _{\eta }(d\bar{\sigma}_{\gamma
}\left\vert \bar{z}_{\gamma }\right. )$ is a probability measure on $%
S^{\gamma }$ of the form%
\begin{equation}
\Pi _{\eta }(d\bar{\sigma}_{\gamma }|\bar{z}_{\gamma })=~\mu _{\eta }(d\bar{%
\sigma}_{\eta }|\bar{z}_{\gamma })\otimes \delta _{\bar{z}_{\gamma \setminus
\eta }}(d\bar{\sigma}_{\gamma \setminus \eta }),  \label{spec-kernel}
\end{equation}%
where%
\begin{equation}
\mu _{\eta }(d\bar{\sigma}_{\eta }|\bar{z}_{\gamma }):=Z(\bar{z}_{\gamma
\setminus \eta })^{-1}\mathrm{exp}~\left[ -E_{\eta }(\bar{\sigma}_{\eta
}\left\vert \bar{z}_{\gamma }\right. )\right] \bigotimes\limits_{x\in \eta
}e^{-V(\sigma _{x})}d\sigma _{x}  \label{mu-kernel}
\end{equation}%
is a probability measure on $S^{\eta }$. Here $Z(\bar{z}_{\eta })$ is the
normalizing factor and $\delta _{\bar{z}_{\gamma \setminus \eta }}(d\bar{%
\sigma}_{\gamma \setminus \eta })$ is the Dirac measure on $S^{\gamma
\setminus \eta }$ concentrated on $\bar{z}_{\gamma \setminus \eta }$. The
family\newline
$\left\{ \Pi _{\eta }(d\bar{\sigma}|\bar{z}),\ \eta \in \mathcal{F}(\gamma ),%
\bar{z}\in S^{\gamma }\right\} $ is called the Gibbsian specification (see
e.g. \cite{Geor,Pre}).

A probability measure $\nu $ on $S^{\gamma }$ is said to be a Gibbs measure
associated with the potentials $W$ and $V$ if it satisfies the DLR equation%
\begin{equation}
\nu (B)=\int_{S^{\gamma }}\Pi _{\eta }(B|\bar{z})\nu (d\bar{z}),\quad B\in 
\mathcal{B}(S^{\gamma }),  \label{DLR}
\end{equation}%
for all $\eta \in \mathcal{F}(\gamma )$. For a given $\gamma \in \Gamma (X)$%
, by $\mathcal{G}(S^{\gamma })$ we denote the set of all such measures.

By $\mathcal{G}_{\alpha ,p}(S^{\gamma })\subset \mathcal{G}(S^{\gamma })$ we
denote the set of all Gibbs measures on $S^{\gamma }$ associated with $W$
and $V$, which are supported on $l_{\alpha }^{p}$.

\begin{theorem}
\label{gibbs1}Assume that conditions (\ref{W-est}) and (\ref{V-est}) are
satisfied and $p\in \left[ r,\tau \right] $. Then the set $\mathcal{G}%
_{\alpha ,p}(S^{\gamma })$ is non-empty for any $\alpha \in \mathcal{A}$.
\end{theorem}

\begin{proof}
It follows in a straightforward manner from condition (\ref{logbound}) that 
\begin{equation*}
a_{\gamma ,\rho }(\gamma )=\sum_{x\in \gamma }e^{-\alpha \left\vert
x\right\vert }n_{x}^{p_{1}}\sum_{y\in \gamma _{x}}n_{y}^{p_{1}}<\infty
\end{equation*}%
for any $p_{1},p_{2}\in \mathbb{N}$, which is sufficient for the existence
of $\nu \in \mathcal{G}_{\alpha ,p}(S^{\gamma })$ for any $p\in \left[
r,\tau \right] $, see \cite{KKP} and \cite{DKKP}.
\end{proof}

\begin{remark}
The result of \cite{KKP}, \cite{DKKP} is more refined and states in addition
certain bounds on exponential moments of $v\in \mathcal{G}_{\alpha
,p}(S^{\gamma })$.
\end{remark}

\begin{remark}
\begin{color}{black}%
Conditions of the uniqueness of $v\in \mathcal{G}_{\alpha ,p}(S^{\gamma })$
are known only in the case of configuration $\gamma $ with bounded sequence $%
\left\{ n_{x},\ x\in \gamma \right\} $. Sufficient conditions of
non-uniqueness (phase transition) for Poisson-distributed $\gamma $ are given
in \cite{DKKP}. 
\end{color}%
\end{remark}

\subsection{Construction of the stochastic dynamics \label{sec-stoch-dyn}}

In this section, we will construct a process $\Xi _{t}$ with invariant
measure $\nu \in \mathcal{G}_{\alpha _{\ast },p}(S^{\gamma })$ defined by
interaction potentials $W$ and $V$ as in Example \ref{ex-gibbs}. By Theorem %
\ref{gibbs1}, the set $\mathcal{G}_{\alpha _{\ast },p}(S^{\gamma })$ is not
empty if $p\in \left[ 2,\tau \right] $. Then, according to the general
paradigm, $\Xi _{t}$ will be a solution of the system (\ref{MainSystem})
with the coefficients satisfying the following:

\begin{enumerate}
\item[(1)] the drift coefficient has a gradient form, that
is, $\phi =-\nabla V$ and $\varphi_{x,y}(\sigma_x,\sigma_y)=\nabla_{\sigma_x}W_{x,y}(\sigma_x,\sigma_y)$; moreover, $\phi$  satisfies Conditions (\ref{bound1}) and (\ref%
{cond-diss1}), a typical example is given by 
\begin{equation*}
\phi (\sigma )=-\sigma ^{2n+1}\text{ for any}\ n=1,2,..,
\end{equation*}%
in which case $R=2n+1$ and $\tau =2n+2$, cf. (\ref{bound1});

\item[(2)] for each $x\in \gamma $, the noise is additive, that is, $\Psi
_{x}=id$.
\end{enumerate}

Thus the system (\ref{MainSystem}) obtains the form
\begin{equation*}
d\xi _{x,t}=\GeorgyMarkup{\frac{1}{2}}\left[ \nabla V(\xi _{x,t})+\sum_{y\in \bar{\gamma}%
_{x}}a(x-y)\xi _{y,t}\right] dt+dW_{x,t},\ x\in \gamma .
\end{equation*}%
According to Theorem \ref{Existence}, this system admits a unique strong
solution $\Xi \in \mathcal{R}_{\alpha +}$ for any initial condition $\bar{%
\sigma}_{\gamma }\in l_{\alpha }^{p}$ with arbitrary $\alpha \in \mathcal{A}$
and $p\geq R$.

A standard way of rigorously proving the invariance of $\nu $ would require
dealing with Markov processes and semigroups in nuclear spaces. This
difficulty can be avoided by using the limit transition (\ref{lim-sem}).

\begin{color}{black}%

\begin{theorem}
Assume that $p\in \left[ \max \left\{ 2,R\right\} ,\tau \right] $ and let $%
{\rm T}_{t}$ be the semigroup defined by the process $\Xi _{t}$, cf. (\ref{sem-def}%
). Then any $\nu \in \mathcal{G}_{\alpha _{\ast },p}(S^{\gamma })$ is a
reversible (symmetrizing) measure for ${\rm T}_{t}$, that is, 
\begin{equation*}
\int {\rm T}_{t}f(\bar{\sigma}_{\gamma })g(\bar{\zeta})\nu (d\bar{\sigma}_{\gamma
})=\int f(\bar{\sigma}_{\gamma }){\rm T}_{t}g(\bar{\zeta})\nu (d\bar{\sigma}%
_{\gamma })
\end{equation*}%
for all $f,g\in C_{b}(l_{\alpha _{\ast }+}^{p})$.
\end{theorem}

\begin{proof}
First observe that condition $p\in \left[ \max \left\{ 2,R\right\} ,\tau %
\right] $ ensures that $\mathcal{G}_{\alpha _{\ast },p}(S^{\gamma })\neq
\emptyset $ and semigroup ${\rm T}_{t}$ is well-defined.

Consider the solution $\Xi ^{n}=(\xi _{x}^{n})_{x\in \gamma }$ of the
truncated system (\ref{FinVolSystem}). Its non-trivial part $(\xi
_{x}^{n})_{x\in \Lambda _{n}}$ is a Markov process in $S^{\Lambda _{n}}$. We
denote by ${\rm T}_{t}^{\Lambda _{n}}$ the corresponding Markov semigroup in $%
C_{b}(S^{\Lambda _{n}})$ and observe that, for $f\in C_{b}(S^{\gamma })$ and 
$f_{\bar{z}_{\gamma }}(\bar{\sigma}_{\Lambda _{n}}):=f(\sigma _{\Lambda
_{n}}\times \bar{z}_{\gamma \setminus \Lambda _{n}})$, we have 
\begin{equation*}
{\rm T}_{t}^{n}f(\bar{\sigma}_{\Lambda _{n}}\times \bar{z}_{\gamma \setminus
\Lambda _{n}})={\rm T}_{t}^{\Lambda _{n}}f_{\bar{z}_{\gamma }}(\bar{\sigma}%
_{\Lambda _{n}}).
\end{equation*}%
By standard theory of finite dimensional SDEs, $\mu _{\Lambda _{n}}(d\bar{%
\sigma}_{\Lambda _{n}}|\bar{z}_{\Lambda _{n}})$ given by (\ref{mu-kernel})
is a reversible (symmetrizing) measure for the semigroup ${\rm T}_{t}^{\Lambda
_{n}}$. Thus we have%
\begin{multline*}
\int {\rm T}_{t}^{n}f(\bar{\sigma}_{\Lambda _{n}}\times \bar{z}_{\gamma \setminus
\Lambda _{n}})g(\bar{\sigma}_{\Lambda _{n}}\times \bar{z}_{\gamma \setminus
\Lambda _{n}})\mu _{\Lambda _{n}}(d\bar{\sigma}_{\Lambda _{n}}|\bar{z}%
_{\gamma }) \\
=\int f(\bar{\sigma}_{\Lambda _{n}}\times \bar{z}_{\gamma \setminus \Lambda
_{n}}){\rm T}_{t}^{n}g(\bar{\sigma}_{\Lambda _{n}}\times \bar{z}_{\gamma \setminus
\Lambda _{n}})\mu _{\Lambda _{n}}(d\bar{\sigma}_{\Lambda _{n}}|\bar{z}%
_{\gamma })
\end{multline*}%
for any $\bar{z}_{\gamma }$. The latter implies in turn that%
\begin{multline*}
\int {\rm T}_{t}^{n}f(\bar{\sigma}_{\gamma })g(\bar{\sigma}_{\gamma })\Pi
_{\Lambda _{n}}(d\bar{\sigma}_{\gamma }|\bar{z}_{\gamma })=\int {\rm T}_{t}^{n}f(%
\bar{\sigma}_{\gamma })g(\bar{\sigma}_{\gamma })~\mu _{\Lambda _{n}}(d\bar{%
\sigma}_{\Lambda _{n}}|\bar{z}_{\gamma })\otimes \delta _{\bar{z}_{\gamma
\setminus \Lambda _{n}}}(d\bar{\sigma}_{\bar{\eta}}) \\
=\int {\rm T}_{t}^{n}f(\sigma _{\Lambda _{n}}\times \bar{z}_{\gamma \setminus
\Lambda _{n}})g(\sigma _{\Lambda _{n}}\times \bar{z}_{\gamma \setminus
\Lambda _{n}})\mu _{\Lambda _{n}}(d\bar{\sigma}_{\Lambda _{n}}|\bar{z}%
_{\gamma }) \\
=\int f(\bar{\sigma}_{\gamma }){\rm T}_{t}^{n}g(\bar{\sigma}_{\gamma })\Pi
_{\Lambda _{n}}(d\bar{\sigma}_{\gamma }|\bar{z}_{\gamma }),
\end{multline*}%
where $\Pi _{\Lambda _{n}}(d\bar{\sigma}_{\gamma }|\bar{z}_{\gamma })=\mu
_{\Lambda _{n}}(d\bar{\sigma}_{\Lambda _{n}}|\bar{z}_{\gamma })\otimes
\delta _{\bar{z}_{\gamma \setminus \Lambda _{n}}}(d\bar{\sigma}_{\bar{\eta}})
$ is the specification kernel, cf. (\ref{spec-kernel}). Integrating with
respect to $\nu (d\bar{z}_{\gamma })$ and applying the DLR equation (\ref%
{DLR}) we see that%
\begin{equation*}
\int {\rm T}_{t}^{n}f(\bar{\sigma}_{\gamma })g(\bar{\sigma}_{\gamma })\nu (d\bar{%
\sigma}_{\gamma })=\int f(\bar{\sigma}_{\gamma }){\rm T}_{t}^{n}g(\bar{\sigma}%
_{\gamma })\nu (d\bar{\sigma}_{\gamma }).
\end{equation*}%
Passing to the limit as $n\rightarrow \infty $ (cf. (\ref{lim-sem})) we
obtain the equality 
\begin{equation*}
\int {\rm T}_{t}f(\bar{\sigma}_{\gamma })g(\bar{\sigma}_{\gamma })\nu (d\bar{\sigma%
}_{\gamma })=\int f(\bar{\sigma}_{\gamma }){\rm T}_{t}g(\bar{\sigma}_{\gamma })\nu
(d\bar{\sigma}_{\gamma }).
\end{equation*}%
as required.
\end{proof}

\end{color}%

\begin{remark}
\begin{color}{black}%
An alternative way to construct stochastic dynamics associated with $\nu
\in \mathcal{G}_{\alpha _{\ast },p}(S^{\gamma })$ is via the theory of
Dirichlet forms. Indeed, $\nu $ satisfies an integration-by-parts formula
and thus defines a classical Dirichlet form, which is a closed bilinear form
in $L^{2}(S^{\gamma },\nu )$. The generator of this form is a non-negative
self-adjoint operator in $L^{2}(S^{\gamma },\nu )$ and thus defines a
strongly continuous semigroup in $L^{2}(S^{\gamma },\nu )$, which, in turn,
defines a Markov process in $S^{\gamma }$ with invariant measure $\nu $
(so-called Hunt process), see \cite{AR} for details. The SDE approach that
we use in our work is, however, more explicit and gives in general better
control on properties of the stochastic dynamics. 
\end{color}%
\end{remark}
\begin{remark}
\AlexMarkup{Observe that pairs $(\gamma ,(\sigma _{x})_{x\in \gamma })$ form the marked
configuration space $\Gamma (X,S)$. For the mathematical formalism of these
spaces and discussion of the existence and uniqueness of Gibbs measures and
phase transitions, see \cite{CDKP}, \cite{DKK} and \cite{RoZass1}, \cite{RoZass2} and
references therein. The paper \cite{RoZass1} considers in particular the case of
marks with values in a path space, which gives a complementary way of
defining and studying infinite dimensional interacting diffusions indexed by
elements of $\gamma \in \Gamma (X)$.}
\end{remark}

\renewcommand\theequation{\thesection.\arabic{equation}}
\appendix

\section{Technical Details \label{sec-auxresults}}
\subsection{Linear operators in the spaces of sequences}
We start with the formulation of a general result from \cite{DaF} on the
existence of (infinite-time) solutions for a special class of linear
differential equations, which extends the so-called Ovsjannikov method, see
e.g. \cite{Deim}.

\begin{definition}
\label{Defovsop} Let $\mathfrak{B}\mathbf{=}\{B_{\alpha }\}_{\alpha \in 
\mathcal{A}}$ be a scale of Banach spaces. A liner operator $%
A:\bigcup\limits_{\alpha \in \mathcal{A}}B_{\alpha }\rightarrow
\bigcup\limits_{\alpha \in \mathcal{A}}B_{\alpha }$ is called an Ovsjannikov
operator of order $q>0$ if $A(B_{\alpha })\subset B_{\beta }$ and there
exists a constant $L>0$ such that 
\begin{equation}
||Ax||_{B_{\beta }}\leq \frac{L}{(\beta -\alpha )^{q}}||x||_{B_{\alpha }},\
x\in B_{\alpha },  \label{Ovs-est}
\end{equation}%
for all $\alpha <\beta \in \mathcal{A}$. The space of such operators will be
denoted by $\mathcal{O}(\mathfrak{B},q)$.
\end{definition}

\begin{theorem}
\label{Th-ovs} \cite[Theorem 3.1 and Remark 3.3]{DaF} Let $A\in \mathcal{O}(%
\mathfrak{B},q)$ with $q<1$. Then, for any $\alpha ,\beta \in \mathcal{A}$
such that $\alpha <\beta $ and $f_{0}\in B_{\alpha }$, there exists a unique
continuous function $f:[0,\infty )\rightarrow B_{\beta }$ with $f(0)=f_{0}$
such that:

\begin{enumerate}
\item[(1)] $f$ is continuously differentiable on $(0,\infty )$;

\item[(2)] $Af(t)\in B_{\beta }$ for all $t\in (0,\infty )$;

\item[(3)] $f$ solves the differential equation 
\begin{equation*}
\dfrac{d}{dt}f(t)=Af(t),\quad t>0.
\end{equation*}
\end{enumerate}

Moreover, 
\begin{equation}\label{OvsEs}
\lVert f(t)\rVert _{B_{\beta }}\leq \GeorgyMarkup{K_{t}}(\alpha ,\beta )\lVert f_{0}\rVert
_{B_{\alpha }},\quad t>0,
\end{equation}%
where $\GeorgyMarkup{K_{t}}(\alpha ,\beta ):=\sum_{n=0}^{\infty }\frac{L^{n}t^{n}}{(\beta
-\alpha )^{qn}}\frac{n^{qn}}{n!}<\infty $.
\end{theorem}

\begin{remark}
\AlexMarkup{Let us remark that estimate (\ref{OvsEs}) generalizes the classical estimate 
$\left\Vert e^{tA}\right\Vert \leq e^{t\left\Vert A\right\Vert }$ for the
exponent of a bounded operator $A$ in a Banach space and does not take into
account possible dissipativity properties of $A$. }
\end{remark}

\begin{remark}
\AlexMarkup{Function $K_{t}(\alpha ,\beta )$ can be estimated in the following way, see \cite{DaF}:%
\begin{equation}
K_{t}(\alpha ,\beta )\leq \sum_{n=0}^{\infty }\frac{L^{n}e^{n}t^{n}}{(\beta-\alpha )^{qn}}\frac{1}{n^{(1-q)n}}.  \label{K-est}
\end{equation}%
The r.h.s. of (\ref{K-est}) is an entire function of order ${\delta} =(1-q)^{-1}$
and type $\sigma =(Le)^{\delta }(e\delta )^{-1}(\beta -\alpha )^{-q\delta }$.
Thus, for any $\varepsilon >0$, there exists $t_{\varepsilon }>0$ such that 
\[
K_{t}(\alpha ,\beta )\leq e^{(\sigma +\varepsilon )t^{\delta +\varepsilon }}%
\text{ for all }t>t_{\varepsilon }.
\]}
\end{remark}

The aim of this section is to give a sufficient condition for the linear operator $Q$, given by an infinite real matrix 
$\{Q_{x,y}\}_{x,y\in \gamma }$, to generate an Ovsjannikov operator in
the scale $\mathcal{L}^{1}$ of spaces of sequences defined by (\ref{l-scale}).

\begin{theorem}
\label{OvsMapTheorem} Assume that $\{Q_{x,y}\}_{x,y\in \gamma }$ is such that for all $x,y\in \gamma $ we have

\begin{itemize}
\item $Q_{x,y}=0$ if $\left\vert x-y\right\vert >\rho $;

\item there exist $C>0$ and $k\geq 1$ such that 
\begin{equation}
|Q_{x,y}|\leq Cn_{x}^{k}.  \label{OvsMapTheoremEqn1}
\end{equation}%
Then $Q\in \mathcal{O}(\mathcal{L}^{1},q)$ for any $q<1$.
\end{itemize}
\end{theorem}

\begin{proof}
Since $Q$ is linear, it is sufficient to show that 
\begin{equation}
\Vert Qz\Vert _{\beta }\leq \frac{L}{(\beta -\alpha )^{q}}\Vert z\Vert
_{\alpha }  \label{OvsMapTheoreme1}
\end{equation}%
for any $\alpha <\beta \in \mathcal{A}$ and $z\in l_{\alpha }^{1}$. By the
definition of the norm in $l_{\beta }^{1}$ we have 
\begin{equation*}
\Vert Qz\Vert _{\beta }=\sum_{x\in \gamma }e^{-\beta |x|}\bigg{|}\sum_{y\in
\gamma }Q_{x,y}z_{y}\bigg{|}.
\end{equation*}%
Now, using estimate (\ref{OvsMapTheoremEqn1}) we see that 
\begin{align}
\Vert Qz\Vert _{\beta }\leq \sum_{x\in \gamma }\sum_{y\in \gamma
}|Q_{x,y}|e^{-\beta |x|}|z_{y}|&\leq e^{\beta \rho }\sum_{x\in \gamma
}\sum_{y\in \bar{\gamma}_{x}}|Q_{x,y}|e^{-\beta |y|}|z_{y}| \nonumber \\
&\leq e^{\beta \rho }\sum_{x\in \gamma }\sum_{y\in \bar{\gamma}%
_{x}}|Q_{x,y}|e^{-(\beta -\alpha )|y|}e^{-\alpha |y|}|z_{y}| \nonumber \\
&\leq e^{\alpha ^{\ast }\rho }\GeorgyMarkup{U}\Vert z\Vert _{\alpha },
\label{OvsMapTheoremEqn3}
\end{align}
because $Q_{x,y}=0$ for $y\GeorgyMarkup{\not\in} \bar{\gamma}_{x}$ and $-|x|\leq -|y|+\rho $
for $y\in \bar{\gamma}_{x}$. Here 
\begin{equation*}
\GeorgyMarkup{U}%
\coloneqq%
\underset{y\in \gamma }{\sup }\sum_{x\in \gamma }|Q_{x,y}|e^{-(\beta -\alpha
)|y|}.
\end{equation*}%
Our next goal is to estimate the constant $\GeorgyMarkup{U}$. Using condition (\ref%
{OvsMapTheoremEqn1}) we see that for all $y\in \gamma $ 
\begin{equation*}
\sum_{x\in \gamma }|Q_{x,y}|e^{-(\beta -\alpha )|y|}\leq C\sum_{x\in
B_{y}}n_{x}^{k}e^{-(\beta -\alpha )|y|}.
\end{equation*}%
Observe that there exist constants $M,N\in \mathbb{N}$ such that 
\begin{equation*}
M<|x|\implies n_{x}\leq N^{1/k}|x|^{q/2k}.
\end{equation*}%
Then, taking into account that $|x|^{q/2}\leq |y|^{q/2}+\rho ^{q/2}$ for $%
x\in B_{y},$ we obtain, assuming without loss of generality that $|y|> M$, that 
\begin{align}
\sum_{x\in B_{y}}n_{x}^{k} &\leq \sum_{\substack{ x\in B_{y}  \nonumber\\ |x|>M}}%
N|x|^{q/2}+\sum_{\substack{ x\in \gamma  \\ |x|\leq M}}n_{x}^{k} \\
& \leq N\sum_{\substack{ x\in B_{y}  \\ |x|>M}}\left( |y|^{q/2}+\rho
^{q/2}\right) +P \ \leq \ Nn_{y}\left( |y|^{q/2}+\rho ^{q/2}\right) +P \nonumber \\
& \leq N^{2}|y|^{q/2}\left( |y|^{q/2}+\rho ^{q/2}\right) +P
\ \leq \ 2N^{2}|y|^{q}+N^{2}\rho ^{q}+P,
\end{align}%
where $P=P(\gamma ,M,q):=\sum_{\substack{ x\in \gamma  \\ |x|\leq M}}%
n_{x}^{q}<\infty .$ Hence for all $y\in \gamma $ we have%
\begin{equation*}
\sum_{x\in \gamma }|Q_{x,y}|e^{-(\beta -\alpha )|y|}\leq C\left(
N^{2}|y|^{q}+N^{2}\rho ^{q}+P\right) e^{-(\beta -\alpha )|y|}\leq
a_{1}+a_{2}|y|^{q}e^{-(\beta -\alpha )|y|}
\end{equation*}%
with $a_{1}=C\left( N^{2}\rho ^{q}+P\right) $ and $a_{2}=CN^{2}$. Now we see
that%
\begin{multline}
\GeorgyMarkup{U}\leq a_{1}+a_{2}\sup \bigg{\{}|y|^{q}e^{-(\beta -\alpha )|y|}\bigg{|}\ y\in
\gamma \bigg{\}} \\
\leq a_{1}+a_{2}\sup \bigg{\{}\bigg{(}he^{-\frac{\beta -\alpha }{q}h}\bigg{)}%
^{q}\bigg{|}\ h>0\bigg{\}} \\
\leq a_{1}+a_{2}\bigg{(}\sup \bigg{\{}he^{-\frac{\beta -\alpha }{q}h}\ %
\bigg{|}\ h>0\bigg{\}}\bigg{)}^{q}.  \label{OvsMapTheoreme3}
\end{multline}%
Hence, we can deduce that function $he^{-\frac{\beta -\alpha }{q}%
h},\,h\in \mathbb{R}$, attains its supremum when $\frac{d}{dh}%
he^{-\frac{\beta -\alpha }{q}h}=0$ that is when $h=\frac{q}{(\beta -\alpha )}
$. Hence it follows from inequality (\ref{OvsMapTheoreme3}) that 
\begin{equation*}
\GeorgyMarkup{U}\leq \frac{a_{1}(\alpha ^{\ast }-\alpha _{\ast })^{q}+a_{2}\left(
e^{-1}q\right) ^{q}}{(\beta -\alpha )^{q}}.
\end{equation*}%
Now, continuing from equation (\ref{OvsMapTheoremEqn3}) we finally see that (%
\ref{OvsMapTheoreme1}) holds with $L=e^{\alpha ^{\ast }\rho }(a_{1}(\alpha
^{\ast }-\alpha _{\ast })^{q}+a_{2}q^{q})$, and the proof is complete.
\end{proof}

\subsection{Comparison theorem and Gronwall-type inequality}
In this section, we prove generalizations of the classical comparison
theorem for differential equations and, as a consequence, a version of the
Gronwall inequality, that works in our scale of Banach spaces of sequences.

Let us consider the linear integral equation 
\begin{equation}
f(t)=\bar{z}+\int_{0}^{t}Qf(s)ds,\quad t\in \mathcal{T},  \label{ct1}
\end{equation}%
in $l_{\alpha ^{\ast }}^{1}$ where $Q\in \mathcal{O}(\mathcal{L}^{1},q)$, $%
q<1$, is a linear operator generated by the infinite matrix $%
\{Q_{x,y}\}_{x,y\in \gamma }$ and $\bar{z}=(z_{x})_{x\in \gamma }\in
l_{\alpha }^{1}$ for some $\alpha <\alpha ^{\ast }$. It follows from Theorem %
\ref{Th-ovs} that this equation has a unique solution $f\in l_{\alpha +}^{1}$%
.

The next result is an extension of the classical comparison theorem to our
framework.

\begin{theorem}[Comparison Theorem]
\label{GronTheorem} \quad \newline
Suppose that $Q_{x,y}\geq 0$ for all $x,y\in \gamma $ and let $g:\mathcal{T}%
\rightarrow l_{\alpha }^{1}$ be a bounded map such that 
\begin{equation*}
g_{x}(t)\leq z_{x}+\bigg{[}\int_{0}^{t}Qg(s)ds\bigg{]}_{x},\quad t\in 
\mathcal{T},\ x\in \gamma .
\end{equation*}%
Then for all $t\in \mathcal{T}$ and all $x\in \gamma $ we have the inequality%
\begin{equation*}
g_{x}(t)\leq f_{x}(t),
\end{equation*}%
where $f=(f_{x})_{x\in \gamma }$ is the solution of (\ref{ct1}) .
\end{theorem}

\begin{proof}
Let $\mathcal{B}_{a}:=\mathcal{B}([0,T],l_{\alpha }^{1}),\ a\in \mathcal{A}$%
, be the Banach space of bounded measurable functions $\mathcal{T\rightarrow 
}l_{\alpha }^{1}$. For any $g\in \mathcal{B}_{\alpha }$ define the function 
\begin{equation*}
\mathcal{I}(g)(t)%
\coloneqq%
\bar{z}+\int_{0}^{t}Qg(s)ds.
\end{equation*}%
It is clear that $\mathcal{I}(g)\in \mathcal{B}_{\alpha +}$, which implies
that the composition power $\mathcal{I}^{n}:\mathcal{B}_{a}\rightarrow 
\mathcal{B}_{a\mathfrak{+}}$ is well-defined. It follows from (the proof of) 
\cite[Theorem 3.1.]{DaF} that%
\begin{equation}
\overbrace{\bigg{[}\lim_{n\rightarrow \infty }\mathcal{I}^{n}(g)\bigg{]}}^{%
\text{in}\ \mathcal{B}([0,T],l_{\beta }^{1})}=f,\ \beta >\alpha .
\label{Lim-g}
\end{equation}%
\begin{color}{black}%
Indeed, 
\begin{equation*}
\mathcal{I}^{n}(g)(t)=\sum_{k=0}^{n-1}\frac{t^{k}}{k!}Q^{k}\bar{z}%
+Q^{n}\int_{0}^{t}...\int_{0}^{t_{n-1}}\GeorgyMarkup{g(t_{n})dt_{n}...dt_{1}}
\end{equation*}%
It was proved in \cite[Theorem 3.1.]{DaF}, cf. formula (3.5), that the
series $\sum_{n=0}^{\infty }\frac{t^{n}}{n!}Q^{n}\bar{z}$ converges
uniformly in any $l_{\beta }^{1}$, $\beta >\alpha $, and $\sum_{n=0}^{\infty
}\frac{t^{n}}{n!}Q^{n}\bar{z}=f(t)$. On the other hand, dividing the
interval $\left[ \alpha ,\beta \right] $ into $n$ intervals of equal length
and using estimate (\ref{Ovs-est}) on each of them, as in \cite[Theorem 3.1.]%
{DaF}, we obtain the bound%
\begin{equation*}
a_{n}:=\left\Vert
Q^{n}\int_{0}^{t}...\int_{0}^{t_{n-1}}\GeorgyMarkup{g(t_{n})dt_{n}...dt_{1}}\right\Vert
_{l_{\beta }^{1}}\leq \sup_{t}\left\Vert g(t)\right\Vert _{l_{\alpha }^{1}}%
\frac{t^{n}}{n!}D^{n}n^{qn}
\end{equation*}%
with $D:=L(\beta -\alpha )^{-q}$. Taking into account that $n!\geq \left( 
\frac{n}{e}\right) ^{n}$ we see that $a_{n}\rightarrow 0,n\rightarrow \infty 
$, which implies (\ref{Lim-g}).%
\end{color}%

We have therefore $\lim_{n\rightarrow \infty }\mathcal{I}%
_{x}^{n}(g)(t)=f_{x}(t)$ for all $x\in \gamma $ and all $t\in \mathcal{T}$.
Hence to conclude the proof it is \GeorgyMarkup{sufficient to fix $x\in \gamma $ and prove by induction that 
 for all $t\in \mathcal{T}$ we have}
\begin{equation}
g_{x}(t)\leq \mathcal{I}_{x}^{n}(g)(t),\ \forall n\in \mathbb{N}.
\label{InduationH1}
\end{equation}%
The case $n=1$ is satisfied by the initial assumption on $g$. Let us now
assume that (\ref{InduationH1}) is true for some $n\geq 1$ and proceed by
considering the following chain of inequalities: 
\begin{multline*}
\mathcal{I}_{x}^{n+1}(g)(t)=z_{x}+\bigg{[}\int_{0}^{t}Q(\mathcal{I}%
^{n}(g)(s))ds\bigg{]}_{x} \\
=z_{x}+\sum_{y\in \gamma }Q_{x,y}\int_{0}^{t}\mathcal{I}_{y}^{n}(g)(s)ds 
\geq z_{x}+\sum_{y\in \gamma }Q_{x,y}\int_{0}^{t}g_{y}(s)ds \\
=z_{x}+\bigg{[}\int_{0}^{t}Q(g(s))ds\bigg{]}_{x}\geq g_{x}(t),
\end{multline*}%
which \GeorgyMarkup{(since t above is arbitrary)} completes the proof.
\end{proof}

\begin{corollary}[Generalized Gronwall inequality]
\label{GronCorollary} Suppose in addition that $z_{x}\geq 0$ for all $x\in
\gamma $. Moreover assume that components of the map $g$ are non-negative
functions, that is, $g_{x}(t)\geq 0$ for all $x\in \gamma $ and all $t\in 
\mathcal{T}$. Then for all $\beta >\alpha $ we have the inequality 
\begin{equation*}
\sum_{x\in \gamma }e^{-\beta |x|}\sup_{t\in \mathcal{T}}g_{x}(t)\leq
\GeorgyMarkup{K_{T}}(\alpha ,\beta )\sum_{x\in \gamma }e^{-\alpha |x|}z_{x},
\end{equation*}%
where $\GeorgyMarkup{K_{T}}(\alpha ,\beta )=\sum_{n=0}^{\infty }\frac{L^{n}T^{n}}{(\beta
-\alpha )^{qn}}\frac{n^{qn}}{n!}<\infty $.
\end{corollary}

\begin{proof}
Using Theorem \ref{GronTheorem}, we see that for all $x\in \gamma $ and all $%
t\in \mathcal{T}$ we have 
\begin{equation*}
g_{x}(t)\leq z_{x}+\bigg{[}\int_{0}^{t}Q(g(s))ds\bigg{]}_{x}\leq z_{x}+%
\bigg{[}\int_{0}^{t}Q(f(s))ds\bigg{]}_{x}.
\end{equation*}%
Since functions $g$ and therefore $f$ are non-negative we see that for all $%
x\in \gamma $ 
\begin{equation*}
\sup_{t\in \mathcal{T}}g_{x}(t)\leq z_{x}+\bigg{[}\int_{0}^{T}Q(f(s))ds%
\bigg{]}_{x}=f_{x}(T).
\end{equation*}%
Hence it follows that 
\begin{equation*}
\sum_{x\in \gamma }e^{-\beta |x|}\sup_{t\in \mathcal{T}}g_{x}(t)\leq
\sum_{x\in \gamma }e^{-\beta |x|}f_{x}(T)\leq \Vert f(T)\Vert _{l_{\beta
}^{1}}.
\end{equation*}%
The right-hand side of the inequality above can be estimated using \cite[%
Theorem 3.1]{DaF}, cf. Theorem \ref{Th-ovs}. In particular, we get 
\begin{equation*}
\Vert f(T)\Vert _{l_{\beta }^{1}}\leq \sum_{n=0}^{\infty }\frac{L^{n}T^{n}}{%
(\beta -\alpha )^{qn}}\frac{n^{qn}}{n!}\Vert \GeorgyMarkup{\bar{z}}\Vert
_{l_{\alpha }^{1}}<\infty .
\end{equation*}%
Hence letting $\GeorgyMarkup{K_{T}}(\alpha ,\beta )=\sum_{n=0}^{\infty }\frac{L^{n}T^{n}}{%
(\beta -\alpha )^{qn}}\frac{n^{qn}}{n!}$ we see that the proof is complete.
\end{proof}

\begin{lemma}
\label{gron111}Consider a bounded measurable map $\varrho :\mathcal{%
T\rightarrow }l_{\alpha }^{1}$, $\alpha\in\mathcal{A}$, and assume that its
components satisfy the inequality 
\begin{equation}
\varrho _{x}(t)\leq Bn_{x}^{k}\sum_{y\in \bar{\gamma}_{x}}\int_{0}^{t}%
\varrho _{y}(s)ds+b_{x},\ t\in \mathcal{T},\ x\in \gamma ,  \label{ineq111}
\end{equation}%
for some constants $B > 0$ and $k \geq 1$ and $b:=(b_{x})_{x\in \gamma }\in $ $l_{\alpha
}^{1}$, $b_{x}\geq 0$. Then we have the estimate%
\begin{equation}
\sum_{x\in \gamma }e^{-\beta |x|}\sup_{t\in \mathcal{T}}\varrho _{x}(t)\leq
\GeorgyMarkup{K_{T}}(\alpha ,\beta )\sum_{x\in \gamma }e^{-\alpha |x|}\GeorgyMarkup{b_{x}}
\label{ineq222}
\end{equation}%
for any $\beta >\alpha $, with $\GeorgyMarkup{K_{T}}(\alpha ,\beta )=\sum_{n=0}^{\infty }\frac{%
L^{n}T^{n}}{(\beta -\alpha )^{qn}}\frac{n^{qn}}{n!}<\infty $, cf. Theorem %
\ref{Th-ovs}.
\end{lemma}

\begin{proof}
Inequality (\ref{ineq111}) can be rewritten in the form 
\begin{equation*}
\varrho _{x}(t)\leq \GeorgyMarkup{\sum_{y\in \bar{\gamma} }}Q_{x,y}\int_{0}^{t}\varrho
_{y}(s)ds+b_{x},\ t\in \mathcal{T,}
\end{equation*}%
where 
\begin{equation*}
Q_{x,y}=%
\begin{cases}
Bn_{x}^{k}, & |x-y|\leq \rho , \\ 
0, & |x-y|>\rho%
\end{cases}%
\end{equation*}%
for all $x\in \gamma .$ We have $\varrho \in \mathcal{B}(\mathcal{T}%
,l_{\alpha }^{1})$, and $|Q_{x,y}|\leq Bn_{x}^{k}$. Therefore using Theorem %
\ref{OvsMapTheorem} we conclude that for any $q\in (0,1)$ matrix $\left(
Q_{x,y}\right) $ generates an Ovsjannikov operator of order $q$ on $\mathcal{%
L}^{1}$. Therefore we can now use Corollary \ref{GronCorollary} to conclude
that (\ref{ineq222}) holds.
\end{proof}

\subsection{Estimates of the solutions \label{sec-est-sol}}
We start with the following auxiliary result.

\begin{lemma}
\label{DriftLemma} Suppose that $\sigma _{1},\sigma _{2}\in \mathbb{R}$ and $%
Z_{1},Z_{2}\in S^{\gamma }$. Then for all $x\in \gamma $ we have the
following inequalities: 
\begin{align*}
|\Psi _{x}(Z_{1})-\Psi _{x}(Z_{2})|& \leq
M\GeorgyMarkup{(n_{x}+1)}|z_{1,x}-z_{2,x}|+M\sum_{y\in\gamma_{x}}|z_{1,y}-z_{2,y}|, \\[0.01in]
|\GeorgyMarkup{\Psi(0)}|& \leq \GeorgyMarkup{Mn_{x}},
\end{align*}%
and
\begin{align}
|\Phi _{x}(Z_{1})| &\leq c(1+|z_{1,x}|^{R}) + \GeorgyMarkup{\bar{a}n_{x}(1+2|z_{1,x}|)
 + \bar{a}\sum_{y\in \gamma_{x}}|z_{1,y}|,} \nonumber \\
(z_{1,x}-z_{2,x})(\Phi _{x}(Z_{1})-\Phi _{x}(Z_{2}))
&\leq (b+\frac{1}{2} \GeorgyMarkup{+4\bar{a}^{2}n_{x}^{2}})(z_{1,x}-z_{2,x})^{2}+\frac{1}{2}\bar{a}%
^{2}n_{x}\sum_{y\in \gamma _{x}}(z_{1,y}-z_{2,y})^{2}, \nonumber
\end{align}%
where constants $M,c,b$ and $\bar{a}$ are defined in Assumption \ref{mainass}%
.
\end{lemma}

\begin{proof}
The proof can be obtained by a direct calculation using assumptions on $\Phi $ and $%
\Psi $ stated in Section \ref{sec-stochsystem}.
\end{proof}

Let us fix $\alpha \in \mathcal{A}$ and consider two processes $\Xi
_{t}^{(1)}=\left( \xi _{x,t}^{(1)}\right) _{x\in \gamma }$ and $\Xi
_{t}^{(2)}=\left( \xi _{x,t}^{(2)}\right) _{x\in \gamma }$, $\Xi ^{(1)},\Xi
^{(2)}\in \mathcal{R}_{\alpha +}^{p}$, \GeorgyMarkup{with initial values $\Xi
_{0}^{1},\Xi _{0}^{2}\in L_{\alpha }^{p}$.}

\begin{lemma}
\label{two-bound} \GeorgyMarkup{Let $p\geq 2, x\in \gamma $ be fixed and assume that $%
\mathbb{R}$-valued processes $\xi _{x,t}^{(1)}, \xi _{x,t}^{(2)}$ satisfy equation (\ref{MainSystem}). Then there exist
universal constants $B,C_{1}$and $C_{2}^{x}$} such that%
\begin{equation}
\mathbb{E}|\xi _{x,t}^{(1)}|^{p}\leq \GeorgyMarkup{\mathbb{E}}\left\vert \xi _{x,0}^{(1)}\right\vert
^{p}+C_{1}n_{x}^{2}\sum_{y\in \bar{\gamma}_{x}}\int_{0}^{t}\mathbb{E}|
\GeorgyMarkup{\xi_{y,s}^{(1)}}|^{p}ds+\GeorgyMarkup{C_{2}^{x}}  \label{single1}
\end{equation}%
and 
\begin{equation}
\mathbb{E}|\bar{\xi}_{x,t}|^{p}\leq \GeorgyMarkup{\mathbb{E}}\left\vert \bar{\xi}_{x,0}\right\vert
^{p}+Bn_{x}^{2}\sum_{y\in \bar{\gamma}_{x}}\int_{0}^{t}\mathbb{E}|\bar{\xi}%
_{y,s}|^{p}ds,  \label{pair1}
\end{equation}%
for all$\ t\in \mathcal{T}$, where $\bar{\xi}_{x,t}:=\xi _{x,t}^{(1)}-\xi
_{x,t}^{(2)}$. The constants $B,C_{1}$ and $C_{2}$ are independent of the
processes $\Xi ^{(1)},\Xi ^{(2)}$ and $x\in \gamma $. \GeorgyMarkup{Moreover $\bar{C}_{2}\coloneqq\{C_{2}^{x}\}_{x\in\gamma}\in l_{\alpha}^{p}$.}
\end{lemma}

\begin{proof}
We remark that in this proof all inequalities hold for all $t\in\timeint$ and $\mathbb{P}-a.s.$, that is on the same same set of measure 1. We now start with the proof of inequality (\ref{single1}). Using Itô Lemma we
see that if $x\in \Lambda _{n}$ then for all $t\in \mathcal{T}$ 
\begin{multline*}
|\xi _{x,t}^{(1)}|^{p}=|\GeorgyMarkup{\xi _{x,0}^{(1)}}|^{p}+p\int_{0}^{t}(\xi
_{x,s}^{(1)})^{p-1}\Phi _{x}(\Xi _{s}^{(1)})ds \\
+\frac{(p-1)p}{2}\int_{0}^{t}(\xi _{x,s}^{(1)})^{p-2}(\Psi _{x}(\Xi
_{s}^{(1)}))^{2}ds \\
+p\int_{0}^{t}(\xi _{x,s}^{(1)})^{p-1}\Psi _{x}(\Xi _{s}^{(1)})dW_{x}(s).
\end{multline*}%
Now from assumptions (\ref{bound1}) and (\ref{cond-diss1}) and Lemma \ref%
{DriftLemma} we can deduce that for all $t\in \mathcal{T}$ 
\begin{multline*}
(\xi _{x,s}^{(1)})^{p-1}\Phi _{x}(\Xi _{s}^{(1)})=(\xi
_{x,s}^{(1)})^{p-2}(\xi _{x,s}^{(1)})\Phi _{x}(\Xi _{s}^{(1)}) \\
\leq \left\vert \xi _{x,s}^{(1)}\right\vert ^{p-2}\left[ (b+\frac{1}{2} \GeorgyMarkup{+4\bar{a}^{2}n_{x}^{2}})|\xi
_{x,s}^{(1)}|^{2}+\frac{1}{2}\Tilde{a}_{x}^{2}\sum_{y\in \gamma _{x}}|\xi
_{y,s}^{(1)}|^{2}+\left\vert \xi _{x,s}^{(1)}\GeorgyMarkup{(\phi(0) + \bar{a}n_{x})}\right\vert \right] \\
\leq (b+\frac{1}{2} \GeorgyMarkup{+4\bar{a}^{2}n_{x}^{2}})|\xi _{x,s}^{(1)}|^{p}+\frac{1}{2}\Tilde{a}_{x}^{2}|\xi
_{x,s}^{(1)}|^{p-2}\sum_{y\in \gamma _{x}}|\xi _{y,s}^{(1)}|^{2}+\left\vert
\xi _{x,s}^{(1)}\right\vert ^{p-1}\GeorgyMarkup{(c + \bar{a}n_{x})} \\
\leq (b+\frac{1}{2} \GeorgyMarkup{+4\bar{a}^{2}n_{x}^{2}})|\xi _{x,s}^{(1)}|^{p}+\frac{1}{2}\Tilde{a}%
_{x}^{2}n_{x}\max_{y\in \bar{\gamma}_{x}}|\xi _{y,s}^{(1)}|^{p}+(1+|\xi
_{x,s}^{(1)}|)^{p}\GeorgyMarkup{(c + \bar{a}n_{x})}.
\end{multline*}%
where $\Tilde{a}_{x}:=\bar{a}\sqrt{n_{x}}$ and constants $\bar{a},$ $b$ and $%
c$ are defined in Assumption \ref{mainass}. In the last inequality, we used
the simple estimate $C^{p-1}\leq (1+C)^{p-1}\leq (1+C)^{p}$ for any $C>0$,
which holds because $p>1$. Taking into account that $\max_{y\in \bar{\gamma}%
_{x}}|\xi _{y,s}^{(1)}|^{p}\leq \sum_{y\in \bar{\gamma}_{x}}|\xi
_{y,s}^{(1)}|^{p}$ and using inequality $(1+\alpha )^{p}\leq
2^{p-1}(1+\alpha ^{p})$ we arrive at the following: 
\begin{multline*}
(\xi _{x,s}^{(1)})^{p-1}\Phi _{x}(\Xi _{s}^{(1)})\leq (b+\frac{1}{2} \GeorgyMarkup{+4\bar{a}^{2}n_{x}^{2}})|\xi
_{x,s}^{(1)}|^{p} \GeorgyMarkup{+}\\
+\frac{1}{2}\Tilde{a}_{x}^{2}n_{x}\sum_{y\in \bar{\gamma}_{x}}|\xi
_{y,s}^{(1)}|^{p}+2^{p-1}\GeorgyMarkup{(c + \bar{a}n_{x})}+2^{p-1}\GeorgyMarkup{(c + \bar{a}n_{x})}|\xi _{x,s}^{(1)}|^{p} \\
\leq (b+\frac{1}{2}\GeorgyMarkup{+4\bar{a}^{2}n_{x}^{2}} + 2^{p-1}\GeorgyMarkup{(c + \bar{a}n_{x})})|\xi _{x,s}^{(1)}|^{p}+\frac{1}{2}\bar{a}%
^{2}n_{x}^{2}\sum_{y\in \bar{\gamma}_{x}}|\xi _{y,s}^{(1)}|^{p}+2^{p-1}\GeorgyMarkup{(c + \bar{a}n_{x})}.
\end{multline*}%
In a similar way, using assumption (\ref{cond-lip1}) we obtain the estimate 
\begin{multline}
(\xi _{x,s}^{(1)})^{p-2}(\Psi _{x}(\Xi _{s}^{(1)}))^{2} \\
\leq (\xi _{x,s}^{(1)})^{p-2}\left[ 3M^{2}\GeorgyMarkup{(n_{x}+1)}^{2}|\xi
_{x,s}^{(1)}|^{2}+3M^{2}n_{x}\sum_{y\in \gamma _{x}}|\xi
_{y,s}^{(1)}|^{2}+3|\Psi (0)|^{2}\right] \\
\leq 3M^{2}\GeorgyMarkup{(n_{x}+1)}^{2}|\xi _{x,s}^{(1)}|^{p}+3M^{2}n_{x}^{2}\sum_{y\in \bar{%
\gamma}_{x}}|\xi _{y,s}^{(1)}|^{p}+3\GeorgyMarkup{M^{2}n_{x}^{2}}|\xi _{x,s}^{(1)}|^{p-2}
\label{ConditionCConsequence} \\
\leq 3\left( \GeorgyMarkup{M^{2}(n_{x}+1)^{2}}+\GeorgyMarkup{M^{2}n_{x}^{2}2^{p-1}}\right) |\xi
_{x,s}^{(1)}|^{p}+3M^{2}n_{x}^{2}\sum_{y\in \bar{\gamma}_{x}}|\xi
_{y,s}^{(1)}|^{p}+3\GeorgyMarkup{M^{2}n_{x}^{2}}2^{p-1}.
\end{multline}%
Observe that $n_{x}\geq 1$. Thus there exist constants $C_{1}, \GeorgyMarkup{C_{2}^{x}>0}$ such
that 
\begin{equation*}
|\xi _{x,t}^{(1)}|^{p}\leq \left\vert \xi _{x,0}^{(1)}\right\vert
^{p}+C_{1}n_{x}^{2}\sum_{y\in \bar{\gamma}_{x}}\int_{0}^{t}|\GeorgyMarkup{\xi
_{y,s}^{(1)}}|^{p}ds+\GeorgyMarkup{C_{2}^{x}}+p\int_{0}^{t}(\xi _{x,s}^{(1)})^{p-1}\Psi _{x}(\Xi
_{s}^{(1)})dW_{x}(s),
\end{equation*}%
which implies that (\ref{single1}) holds.

The proof of inequality (\ref{pair1}) can be obtained similarly.
Using the relation 
\begin{equation*}
\bar{\xi}_{x,t}=\bar{\xi}_{x,0}+\int_{0}^{t}\left( \Phi _{x}(\Xi
_{s}^{(1)})-\Phi _{x}(\Xi _{s}^{(2)})\right) ds+\int_{0}^{t}\left( \Psi
_{x}(\Xi _{s}^{(1)})-\Psi _{x}(\Xi _{s}^{(2)})\right) dW_{x}(s),
\end{equation*}%
$\ t\in \mathcal{T},$ and applying the Itô Lemma to $|\bar{\xi}_{x,t}|^{p}$
we obtain the inequality
\begin{multline}
|\bar{\xi}_{x,t}|^{p}\leq \left\vert \bar{\xi}_{x,0}\right\vert
^{p}+Bn_{x}^{2}\sum_{y\in \bar{\gamma}_{x}}\int_{0}^{t}|\bar{\xi}%
_{y,s}|^{p}ds  \label{form222} 
+\int_{0}^{t}p(\bar{\xi}_{x,t})^{p-1}\left( \Psi _{x}(\Xi _{s}^{(1)})-\Psi
_{x}(\Xi _{s}^{(2)})\right) dW_{x}(s)
\end{multline}%
for some constant $B>0$, which implies the result. \GeorgyMarkup{Finally, $\bar{C}_{2}\in l_{\alpha}^{p}$ because (see Assumption \ref{mainass}) for some constant $W$ we have $C_{2}^{x} \leq W(1+\log (1+|x|))$ and one can use exponential weight to sum up these terms.}
\end{proof}

\end{document}